\documentclass[10pt]{amsart}
\usepackage{amscd}
\usepackage{amsmath}
\usepackage{amssymb}
\usepackage{latexsym}

\swapnumbers \theoremstyle{definition}
\newtheorem{s}{}[subsection]
\newtheorem{rmk}[s]{Remark}
\theoremstyle{plain}
\newtheorem{thm}[s]{Theorem}
\newtheorem{prop}[s]{Proposition}
\newtheorem{lem}[s]{Lemma}
\newtheorem{sublem}[s]{Sublemma}
\newtheorem{cor}[s]{Corollary}

\newcommand{\frakg}{{\mathfrak g}}

\newcommand{\frakl}{{\mathfrak l}}
\newcommand{\frakm}{{\mathfrak m}}

\newcommand{\fraks}{{\mathfrak s}}
\newcommand{\frakt}{{\mathfrak t}}

\newcommand{\A}{{\mathbb A}}
\newcommand{\C}{{\mathbb C}}
\newcommand{\G}{{\mathbb G}}
\newcommand{\bbP}{{\mathbb P}}
\newcommand{\bbV}{{\mathbb V}}
\newcommand{\Z}{{\mathbb Z}}
\newcommand{\F}{{\mathcal F}}
\newcommand{\I}{{\mathcal I}}
\newcommand{\J}{{\mathcal J}}
\newcommand{\K}{{\mathcal K}}
\newcommand{\calL}{{\mathcal L}}
\newcommand{\calO}{{\mathcal O}}
\newcommand{\calP}{{\mathcal P}}
\newcommand{\V}{{\mathcal V}}
\newcommand{\Gr}{{\mathrm{Gr}}}
\newcommand{\Hom}{{\mathrm{Hom}}}
\newcommand{\ind}{{\mathrm{Ind}}}

\newcommand{\pr}{{\mathrm{pr}}}
\newcommand{\res}{{\mathrm{Res}}}
\newcommand{\spec}{{\mathrm{Spec}}}

\title[Affine Demazure modules and $T$-fixed point subschemes]{\bf Affine Demazure modules and $T$-fixed point subschemes in the affine Grassmannian}
\address{Department of Mathematics, University of California at Berkeley, CA 94720}\email{xinwenz@math.berkeley.edu}
\author{Xinwen Zhu}
\date{October 2007; revised November 2008}
\begin{document}
\begin{abstract}
Let $G$ be a simple algebraic group defined over $\C$ and $T$ be a
maximal torus of $G$. For a dominant coweight $\lambda$ of $G$,
the $T$-fixed point subscheme $(\overline{\Gr}_G^\lambda)^T$ of
the Schubert variety $\overline{\Gr}_G^\lambda$ in the affine
Grassmannian $\Gr_G$ is a finite scheme. We prove that for all
such $\lambda$ if $G$ is of type $A$ or $D$ and for many of them
if $G$ is of type $E$, there is a natural isomorphism between the
dual of the level one affine Demazure module corresponding to
$\lambda$ and the ring of functions (twisted by certain line
bundle on $\Gr_G$) of $(\overline{\Gr}_G^\lambda)^T$. We use this
fact to give a geometrical proof of the Frenkel-Kac-Segal
isomorphism between basic representations of affine algebras of
$A,D,E$ type and lattice vertex algebras.

\end{abstract}
\maketitle

\section*{Introduction}
\subsection{The Frenkel-Kac-Segal Isomorphism}
In their fundamental papers \cite{FK},\cite{Se}, I. Frenkel-Kac
and Segal constructed the bosonic realizations of the basic
representations of simply-laced affine algebras, using the vertex
operators of string theory. Let us first review their theorem.

\begin{s}
Let $\frakg$ be the Lie algebra of a simple algebraic group $G$
over $\C$, and $(\,,\,)$ denote the normalized invariant form on
$\frakg$. The untwisted affine Kac-Moody algebra associated to
$(\frakg,(\,,\,))$ is
\[\hat{\frakg}=\frakg\otimes \C[t,t^{-1}]\oplus \C K\]
where $K$ is central in $\hat{\frakg}$ and
\[[X\otimes t^m,Y\otimes t^n]=[X,Y]\otimes t^{m+n}+m\delta_{m,-n}(X,Y)K\]

Recall that the level $k$ vacuum module of $\hat{\frakg}$ is
defined as
\[\bbV(k\Lambda)=\ind_{\frakg\otimes\C[t]+\C
K}^{\hat{\frakg}}\C\] on which $\frakg\otimes\C[t]$ acts through
the trivial character and $K$ acts by multiplication by $k$.
$\bbV(k\Lambda)$ has a vertex algebra structure. When $k$ is a
positive integer, the vacuum module has a unique irreducible
quotient $L(k\Lambda)$, which is an integrable
$\hat{\frakg}$-module of level $k$. Furthermore, it is a quotient
vertex algebra of $\bbV(k\Lambda)$.

On the other hand, choose $\frakt\subset\frakg$ to be a Cartan
subalgebra of $\frakg$. Let $\iota:\frakt\to\frakt^*$ be the
isomorphism induced by the bilinear form on $\frakt$ incduced from
the normalized invariant form on $\frakg$. The restriction of the
central extension of $\frakg\otimes\C[t,t^{-1}]$ to
$\frakt\otimes\C[t,t^{-1}]$ defines the Heisenberg Lie algebra
$\hat{\frakt}\subset\hat{\frakg}$. Let $R_G\subset\frakt$ be the
coroot lattice of $G$. Define a module over $\hat{\frakt}$ by
\[V_{R_G}=\bigoplus_{\lambda\in R_G}\pi_\lambda\]
where $\pi_\lambda$ is the level one Fock module of $\hat{\frakt}$
with highest weight $\iota\lambda$ (see \ref{Heisenberg}). When
$G$ is simply-laced, there is a unique (up to isomorphism) simple
vertex algebra structure on $V_{R_G}$, corresponding to the
cohomology class $H^2(R_G,\C^*)$ determined by the invariant form.

Now the result of Frenkel-Kac and Segal may be stated as follows
(in the language of vertex algebras).
\end{s}
\begin{thm}\label{KFS-Isom}
If $\frakg$ is a simple Lie algebra of type $A$, $D$ or $E$, then
$L(\Lambda)\cong V_{R_G}$ as vertex algebras.
\end{thm}

In particular, we could obtain a character formula for
$L(\Lambda)$.
\begin{cor}\label{character formula}
$L(\Lambda)$ is isomorphic to $V_{R_G}$ as $\hat{\frakt}$-modules.
\end{cor}
We should point out that we reverse the historical order here. In
fact, this corollary was obtained before the FKS isomorphism (cf.
\cite{Kac} (3.37)), and was used to prove the FKS isomorphism.

\subsection{Main Results}
The goal of the present paper is to interpret the FKS isomorphism
from an algebro-geometrical point of view.

\begin{s}The starting point is the Borel-Weil theorem for affine
Kac-Moody algebras, which was originally proved in \cite{Ku} and
\cite{Ma}. For any algebraic group $G$ over $\C$, let
$\Gr_G=G_\K/G_\calO$ be the affine Grassmannian of $G$, where
$G_\K$ is group of maps from punctural disc to $G$ and $G_\calO$
is the group of maps from disc to $G$. When $G$ is simple,
simply-connected, there is an ample invertible sheaf $\calL_G$ on
$\Gr_G$, which is a generator of the Picard group of $\Gr_G$. The
Borel-Weil type theorem identifies $L(k\Lambda)$ with
$\Gamma(\Gr_G,\calL_G^{\otimes k})^*$. On the other hand, for a
chosen maximal torus $T\subset G$, the affine Grassmannian $\Gr_T$
of $T$ naturally embeds into $\Gr_G$, and $V_{R_G}$ is then
identified with $\Gamma(\Gr_G,\calO_{\Gr_T}\otimes\calL_G)^*$.
Therefore, the question now is whether the natural morphism
\begin{equation}\label{restriction}\calL_G\to\calO_{\Gr_T}\otimes\calL_G\end{equation} induces
an isomorphism between the spaces of their global sections, if $G$
is simply-laced.

It turns out that we can push the question further. Recall that
$\Gr_G$ is stratified by $G_\calO$-orbits, which are parameterized
by dominant coweights. For a dominant coweight $\lambda$, the
corresponding $G_\calO$-orbit is denoted by $\Gr_G^\lambda$. Let
$\overline{\Gr}_G^\lambda$ be the closure of $\Gr_G^\lambda$ in
$\Gr_G$. Since $G$ is simple,
$\Gr_G=\lim\limits_{\longrightarrow}\overline{\Gr}_G^\lambda$.
Then one could ask, whether the restriction of (1) to each
$\overline{\Gr}_G^\lambda$ will still induce an isomorphism
between the spaces of their global sections. Let us reformulate
the question slightly differently. The maximal torus $T$ of $G$
acts on the affine Grassmannian $\Gr_G$ as well as on each
Schubert variety $\overline{\Gr}_G^\lambda$. It will be shown that
the embedding $\Gr_T\subset \Gr_G$ identifies $\Gr_T$ as the
$T$-fixed point subscheme of $\Gr_G$. Therefore,
$\Gr_T\times_{\Gr_G}\overline{\Gr}_G^\lambda$ is the $T$-fixed
point subscheme $(\overline{\Gr}_G^\lambda)^T$ of
$\overline{\Gr}_G^\lambda$. Then the restriction of
(\ref{restriction}) to $\overline{\Gr}_G^\lambda$ becomes
\begin{equation}\label{further restriction}\calO_{\overline{\Gr}_G^\lambda}\otimes\calL_G\to\calO_{(\overline{\Gr}_G^\lambda)^T}\otimes\calL_G
\end{equation}

The main theorem of this paper is
\end{s}
\begin{thm}\label{Main Theorem}
Let $G$ be a simple, connected (not necessarily simply-connected)
algebraic group. Let $\Gr_G$ be the affine Grassmannian of $G$ and
$\overline{\Gr}_G^\lambda$ be the Schubert variety associated to a
dominant coweight $\lambda$. Fix $T\subset G$ a maximal torus, and
let $(\overline{\Gr}_G^\lambda)^T$ be the $T$-fixed subscheme of
$\overline{\Gr}_G^\lambda$. Let $\calL_G$ be the ample invertible
sheaf on $\Gr_G$, which is the generator of the Picard group of
each connected component. Then for all $\lambda$ if $G$ is of type
$A$ or $D$, and for many of then if $G$ is of type $E$, the
natural morphism (\ref{further restriction}) induces an
isomorphism
\[\Gamma(\overline{\Gr}_G^\lambda,\calL_G)=\Gamma(\overline{\Gr}_G^\lambda,\calO_{(\overline{\Gr}_G^\lambda)^T}\otimes\calL_G)\]
\end{thm}

We expect that the theorem still holds for all $\lambda$ if $G$
is of type $E$, for the FKS isomorphism holds for all
simply-laced algebraic groups. However, so far only partial
results are obtained for this type, see \ref{type E,I}-\ref{type
E,II}. Observe that the theorem could not hold for
non-simply-laced algebraic groups.

\begin{s}
One of the applications of the above theorem is to the study of
the singularities of Schubert varieties in the affine
Grassmannian. In principle, for any algebraic variety with a torus
action, the singularity at a fixed point under the action is
reflected by the local ring of the fixed point subscheme at that
point. In particular, we will prove that
\end{s}

\begin{cor}\label{Smooth locus}
Let $G$ be a simple algebraic group of type $A$ or $D$. Then the
smooth locus of $\overline{\Gr}_G^\lambda$ is $\Gr_G^\lambda$.
\end{cor}

In fact, this holds for any simple algebraic group, as is shown in
\cite{EM} and \cite{MOV}. However, since we will deduce it as a
corollary of our main theorem, we will confine ourselves to
algebraic groups of type $A$ or $D$. Remark that the same
statement of our main theorem for algebraic groups of type $E$
will imply the above corollary for them as well.

\begin{s}
The main application of Theorem \ref{Main Theorem} is an
alternative proof of the FKS isomorphism, which is applicable to
all simply-laced algebraic groups. More precisely, we will prove a
geometrical version of the isomorphism using Beilinson and
Drinfeld's idea of factorization algebras (cf. \cite{BD1}). In
addition, the main theorem also allows us to identify the modules
over the basic representations with the modules over the lattice
vertex algebras.

Assume that $G$ is of adjoint type. Let $\Lambda_G$ be the
coweight lattice of $G$. It is known (cf. \cite{Do}) that the
simple modules over $V_{R_G}$ are labelled by
$\gamma\in\Lambda_G/R_G$. Let $V^\gamma_{R_G}$ denote the one
corresponding to $\gamma\in\Lambda_G/R_G$. As
$\hat{\frakt}$-modules,
\[V^\gamma_{R_G}=\bigoplus_{\lambda\in\gamma+R_G}\pi_\lambda\]
It is known that for any $\gamma\in\Lambda_G/R_G$, there is a
unique minuscule fundamental coweight $\omega_{i_\gamma}$, which
is a representative of the coset $\gamma$. Then the simple
$\hat{\frakg}$-module $L(\Lambda+\iota\omega_{i_\gamma})$ of
highest weight $(\Lambda+\iota\omega_{i_\gamma})$ has a module
structure over $L(\Lambda)$. We have
\end{s}
\begin{thm}\label{Vertex operator modules} Under the identification of $L(\Lambda)\cong V_{R_G}$ of vertex algebras, we have an
isomorphism \[L(\Lambda+\iota\omega_{i_\gamma})\cong
V^\gamma_{R_G}\] as modules over them.
\end{thm}

\subsection{New perspectives of the FKS isomorphism}Our
geometrical interpretation of the FKS isomorphism brings us more
insight into this fundamental theorem. We briefly indicate some of
the new perspectives in this introduction. Details and further
discussions will appear in \cite{Zhu}. The rest of the paper is
independent of this subsection.

\begin{s}
As we point out, the FKS isomorphism (more precisely, Corollary
\ref{character formula}) amounts to
\[\Gamma(\Gr_G,\calL_G)\cong\Gamma(\Gr_T,\calL_G|_{\Gr_T})\]
The crucial observation is that instead of interpreting $\Gr_T$ as
the $T$-fixed point subcheme of $\Gr_G$, we can also interpret it
as the affine Springer fiber (cf. \cite{KL} and \cite{GKM}) in
$\Gr_G$ corresponding to the regular semisimple element $\rho$ in
$\frakg(\K)$. This naturally leads to the following consideration.
Let $u\in\frakg(\K)$ be any regular semi-simple element and let
$\Gr_{G,u}$ be the corresponding affine Springer fiber. We can ask
the similar question whether the natural map
\begin{equation}\label{springer fiber}\Gamma(\Gr_G,\calL_G)\to\Gamma(\Gr_{G,u},\calL_G|_{\Gr_{G,u}})\end{equation}
is an isomorphism in the case when $G$ is simply-laced.
Furthermore, whether the restriction of the above map to each
Schubert variety $\overline{\Gr}_G^\lambda$
\begin{equation}\label{springer fiber cap schubert variety}\Gamma(\overline{\Gr}_G^\lambda,\calL_G)\to\Gamma(\overline{\Gr}_G^\lambda\cap\Gr_{G,u},\calL_G|_{\overline{\Gr}_G^\lambda\cap\Gr_{G,u}})\end{equation}
is an isomorphism. In the following, we will see why these are
interesting questions to ask.
\end{s}

\begin{s}
Let $J_u$ be the centralizer of $u$ in $G_\K$. This is a maximal
torus of $G_\K$, which is usually non-split. It is well-known (cf.
\cite{KL}) that the conjugacy classes maximal tori in $G_\K$ are
parameterized by conjugacy classes of the Weyl group of $G$. For
example, if $u=\rho$, then $J_\rho$ is a split torus which
corresponds to the identity element in the Weyl group while if
$u=p=\sum e_{\check{\alpha}_i}+tf_{\check{\theta}}$, $J_p$
corresponds to the Coxeter element.

The $\G_m$-central extension of $G_\K$ gives rise to a
$\G_m$-central extension of $J_u$, called the Heisenberg group,
and denoted by $\hat{J}_u$. For example, $\hat{J}_\rho$ is called
the homogeneous Heisenberg group and $\hat{J}_p$ is called the
principal Heisenberg group. The dual statement of (\ref{springer
fiber}) says that $L(\Lambda)$ is isomorphic to
$\Gamma(\Gr_{G,u},\calL_G|_{\Gr_{G,u}})^*$ as $\hat{J}_u$-modules.
If $\Gr_{G,u}$ is 0-dimensional, one can show that
$\Gamma(\Gr_{G,u},\calL_G|_{\Gr_{G,u}})^*$ is an irreducible
$\hat{J}_u$ module. Therefore, (\ref{springer fiber}) implies that
$L(\Lambda)$ remains irreducible as a $\hat{J}_u$-module, for any
type of Heisenberg group $\hat{J}_u$. In particular, if $u=\rho$,
this is the FKS isomorphism and if $u=p$, this is the so-called
principal realization of the basic representation, which was
proved in \cite{KKLW}. (We remark here that there is also a
statement for $u=p$ similarly to Theorem \ref{Main Theorem}, which
can be obtained by the similarly method in this paper.) However,
the types of such Heisenberg groups are parameterized by conjugacy
classes of the Weyl group of $G$. We thus obtain for any conjugacy
class of the Weyl group, a realization of the basic
representation. Indeed, this has already been considered in
\cite{KP}.
\end{s}

\begin{s}
Now we give another interesting indication of (\ref{springer
fiber}). For this, we should first recall the finite dimensional
story. Let $\check{\lambda}$ be a dominant weight of $G$. We
denote $P_{\check{\lambda}}$ be the parabolic subgroup of $G$ such
that the stablizer $W_{\check{\lambda}}$ of $\check{\lambda}$ in
the Weyl group $W$ of $G$ is the same as the Weyl group
$W_{P_{\check{\lambda}}}$ of $P_{\check{\lambda}}$. Let
$\calP_{\check{\lambda}}$ be the variety of parabolic subgroups of
$G$ of type $P_{\check{\lambda}}$. Let $\calO(\check{\lambda})$ be
the invertible sheaf on $\calP_{\check{\lambda}}$ such that
$\Gamma(\calP_{\check{\lambda}},\calO(\check{\lambda}))^*$ the
irreducible $G$-module of highest weight $\check{\lambda}$. Let
$\xi\in\frakg$ be any regular element, and
$\calP_{\check{\lambda}}^\xi$ be the corresponding Springer fiber
(which is a finite subscheme of $\calP_{\check{\lambda}}$). One
can prove that the natural map
\begin{equation}\label{f.d.}\Gamma(\calP_{\check{\lambda}},\calO(\check{\lambda}))\to\Gamma(\calP_{\check{\lambda}}^\xi,\calO(\check{\lambda})|_{\calP_{\check{\lambda}}^\xi})\end{equation}
is always surjective. It is an isomorphism if and only if
$\check{\lambda}$ is minuscule weight of $G$. It is clear that
(\ref{springer fiber}) is the affine counterpart of the above
statement. Therefore, it suggests to us to call $\Lambda$ (as well
as $\Lambda+\omega_{i_\gamma}$) a minuscule weight of
$\hat{\frakg}$ when $\frakg$ is simply-laced. Observe that in
non-simply-laced case, it seems no weights of $\hat{\frakg}$
should be called minuscule. (One should consider the twisted
affine algebras.)

Let us also point out that in finite dimensional story, the
isomorphism (\ref{f.d.}) for $\check{\lambda}$ minuscule also
reflects the fact that the Schubert variety
$\overline{\Gr}_{^LG}^{\check{\lambda}}$ for the Langlands dual
group $^LG$ is smooth. From (\ref{springer fiber}), one also
expects that the Schubert variety corresponding to $\Lambda$ in
the double affine Grassmannian of the Langlands dual group of
$\hat{G}_\K$ should also be smooth in some sense.
\end{s}

\subsection{Contents} The paper is organized as follows.

In \S \ref{Recollection}, we collect various facts related to
affine Grassmannians that are needed in the sequel. In \S
\ref{Affine Grassmannians}, we recall some geometry of affine
Grassmannians. In \S \ref{flat deg}, we construct a flat
degeneration of
$\overline{\Gr}_G^\lambda\times\overline{\Gr}_G^\mu$ to
$\overline{\Gr}_G^{\lambda+\mu}$. While it is not difficult to
produce such a flat family, it is not trivial to prove that the
special fiber is reduced. It is based on the the tensor structure
of affine Demazure modules, first discovered in \cite{FL} (Theorem
1 of \emph{loc. cit.}) by combinatoric methods. We reprove their
result here in a purely algebro-geometrical way (Theorem
\ref{tensor structure}). In \S \ref{fp}, we show that the
embedding $\Gr_T\to \Gr_G$ identifies $\Gr_T$ as the $T$-fixed
point subscheme of $\Gr_G$. In \S \ref{BW}, we prove a Borel-Weil
type theorem for the nonneutral components of the affine
Grassmannians.

\S \ref{main proof} is devoted to the proof of our main theorem,
Theorem \ref{Main Theorem}. In \S \ref{first reduction}, we reduce
the full theorem to the special cases where $\lambda$ is a
fundamental coweight, so that the geometry of the corresponding
Schubert variety is relatively simple. In \S \ref{proof}, we prove
the theorem for fundamental coweights of algebraic groups of type
$A$ and $D$. We will also have a brief discussion on some partial
results for the algebraic groups of type $E$. In \S \ref{simple
app}, we prove \ref{Smooth locus} as a simple application of our
main theorem.

In \S \ref{Boson}, we return to our motivation of the paper. In \S
\ref{Hei}, we deduce from the main theorem that
$L(\Lambda+\iota\omega_{i_\gamma})$ is isomorphic to
$V^\gamma_{R_G}$ as $\hat{\frakt}$-modules if $G$ is of type
$A,D,E$. Then we recall some basic ingredients of lattice
factorization algebras and affine Kac-Moody factorization algebras
in \S \ref{ver}. Finally, we reprove the FKS isomorphism as well
as Theorem \ref{Vertex operator modules} in \S \ref{FKS}.

\subsection{Notation and Conventions} Throughout this paper, we will work over the base field $\C$.
However, all results remain true over any algebraically closed
field of characteristic 0. A $\C$-algebra will always be assumed
to be associative, commutative and unital.

If $X$ is a scheme (or a space) over some base $S$, and $S'\to S$
is a morphism, we denote $X_{S'}(\mbox{ or }X|_{S'})=X\times_S S'$
the base change of $X$ to $S'$.

If $G$ is a group scheme, $\F$ is a $G$-torsor over some base, and
$X$ is a $G$-scheme, we write $\F\times^G X$ for the associated
product.

Let $\K=\C((t))$ and $\calO=\C[[t]]$. If $G$ is an algebraic
group, we will denote by $G_\calO$ the group scheme whose
$\C$-points are $G(\calO)$ and by $G_\K$ the ind-group whose
$\C$-points are $G(\K)$. The neutral connected component of $G_\K$
is denoted by $G^0_\K$.

If $G$ is a reductive group, we will denote $\Lambda_G$ (resp.
denote $R_G$, resp. choose $\Lambda_G^+$) the coweight lattice
(resp. the coroot lattice, resp. dominant coweights) of $G$ and
$\check{?}$ the dually named object corresponding to $?$. If
$\check{\nu}\in\check{\Lambda}_G^+$ is a dominant weight, we
denote $V^{\check{\nu}}$ the irreducible representation of $G$ of
highest weight $\check{\nu}$. However, if $G$ is a torus, then
$V^{\check{\nu}}$ is 1-dimensional, and will be denoted as
$\C^{\check{\nu}}$. The Weyl group of $G$ is denoted by $W$.

If $\frakg$ is a simple Lie algebra over $\C$, we will denote by
$I$ the set of vertices of the Dynkin diagram of $\frakg$. Then
fundamental coweight (resp. fundamental weight, resp. simple
coroot, resp. simple root) of $\frakg$ corresponding to $i\in I$
is denoted (or rather, chosen) to be $\omega_i$ (resp.
$\check{\omega}_i$, resp. $\alpha_i$, resp. $\check{\alpha}_i$).
The highest root of $\frakg$ is denoted by $\check{\theta}$ and
the corresponding coroot by $\theta$. Remark that $\theta$ is
usually not the highest coroot. The invariant bilinear form
$(\,,\,)$ on $\frakg$ is normalized so that the corresponding
bilinear form $(\,,\,)^*$ on the dual of Cartan subalgebra of
$\frakg$ satisfies that the square of the length of the long root
is 2. Then the invariant form induces an isomorphism from the
Cartan subalgebra to its dual, which is denoted by $\iota$. The
untwisted affine Kac-Moody algebra associated to
$(\frakg,(\,,\,))$ is denoted by $\hat{\frakg}$. Denote by $i_0$
the extra vertex in the affine Dynkin diagram. The fundamental
weight of $\hat{\frakg}$ corresponding to $i_0$ is denoted by
$\Lambda$. Denote by $K$ the central element in $\hat{\frakg}$
such that $\Lambda(K)=1$. Denote by $\check{\delta}$ the imaginary
root of $\hat{\frakg}$ such that
$\check{\alpha}_0=\check{\delta}-\check{\theta}$ is the simple
root corresponding to $i_0$.

\subsection{Acknowledgement}
The author is very grateful to his advisor, Edward Frenkel, for
many stimulating discussions and careful reading of the early
draft. Without his encouragement, this paper would have never been
written up. The author also thanks Joel Kamnitzer and Zhiwei Yun
for very useful discussions.

\section{Affine Grassmannians}\label{Recollection}
We recall the definition and the basic properties of the affine
Grassmannian associated to an algebraic group in \S 1.1. Then we
study the tensor structure of affine Demazure modules in \S 1.2,
the fixed point subscheme of the affine Grassmannian in \S 1.2 and
the Borel-Weil theorem in \S 1.3.
\subsection{Affine Grassmannians}\label{Affine Grassmannians}
\begin{s}
The affine Grassmannian $\Gr_G$ of \emph{any} affine algebraic
group $G$ is defined to be the \emph{fppf} quotient
$G_\K/G_\calO$. Without loss of generality, we could and will
assume that $G$ is connected in the rest of the paper. We collect
here some facts most relevant to our application in this paper,
and refer to \cite{BD2} Section 4.5, \cite{BLS} and \cite{LS} for
a general discussion of affine Grassmannians.
\end{s}

\begin{thm}\label{basic fact}
(1)$\Gr_G$ is an ind-scheme of ind-finite type, and ind-projective
if $G$ is reductive;

(2)$\Gr_G$ is reduced if and only if $\Hom(G,\G_m)=0$.

(3)There is a bijection $\pi_0(\Gr_G)=\pi_1(G)$, and different
connected components of $\Gr_G$ are isomorphic as ind-schemes.

(4)If $G$ is simple and simply-connected, then
$\mbox{Pic}(\Gr_G)=\Z$.
\end{thm}

\begin{rmk}\label{formal part}
In the paper, we will consider the scheme structure of $\Gr_G$. By
(2) of above theorem, if $G$ is a connected semisimple algebraic
group, $\Gr_G$ is reduced. On the other hand, if $G=T$ is a torus,
then $(\Gr_T)_{red}=X_*(T):=\Hom(\G_m,T)$ is a discrete space, but
$\Gr_T$ itself is nonreduced. Indeed, the connected component of
$1\in T(\K)/T(\calO)$ is the formal group with Lie algebra
$\frakt(\K)/\frakt(\calO)$, where $\frakt$ is the Lie algebra of
$T$.
\end{rmk}

\begin{s}\label{orbits of Gr}
Assume that $G$ is reductive. We choose $T\subset G$, a maximal
torus. Recall that $\Gr_G$ is stratified by the $G_\calO$-orbits,
labelled by $\Lambda_G^+$. For any $\lambda\in
\Lambda_G=\Hom(\G_m,T)$, a choice of a uniformizer $t\in\calO$
determines a closed point $t^\lambda\in G_\K$. Let $s^\lambda$ be
the image of $t^\lambda$ under the map $G_\K\to \Gr_G$. Then
$s^\lambda$ is a well defined point in $\Gr_G$, which does not
depend on the choice of the uniformizer $t$. Let
$\Gr_G^{\lambda}=G_\calO s^\lambda$. Remark that $\Gr_G^\lambda$
does not depend on the choice of the maximal torus. Then
$\Gr_G(\C)$ is the union of $\Gr_G^\lambda(\C)$ for
$\lambda\in\Lambda_G^+$. Let $\overline{\Gr}_G^{\lambda}$ its
closure in $\Gr_G$. It is known that
$\overline{\Gr}_G^\mu\subset\overline{\Gr}_G^\lambda$ if and only
if $\mu\leq\lambda$. If $\Gr_G$ is reduced, e.g. $G$ is
semisimple, then as ind-schemes,
$\Gr_G\cong\lim\limits_{\longrightarrow}\overline{\Gr}^\lambda_G$,
where the limit is taken over $\lambda\in\Lambda_G^+$.
\end{s}

\begin{s}
Let $G$ be a reductive group. As in \cite{BL2}, there is a moduli
interpretation of the affine Grassmannian $\Gr_G$. Fix a smooth
curve $X$, and a closed point $x\in X$. Denote $X^*=X-\{x\}$. Then
\[\Gr_{G,x}(R)=\{\F \mbox{ a } G\mbox{-torsor on } X_R, \beta: \F_{X^*_R}\to X^*_R\times G \mbox{ a trivialization on } X^*_R\}\]
is represented by an ind-scheme which is isomorphic to $\Gr_G$,
once an isomorphism $\calO_x\cong\calO$ is chosen, where $\calO_x$
is the complete local ring of $x$. Remark that if $T\subset G$ a
maximal torus is chosen, $s^\lambda$ is represented by the pair
$(\F,\beta)$, where $\F$ is a $T$-torsor over $X$, and $\beta$ is
as above, such that for any $\check{\nu}\in\check{\Lambda}_G$, the
induced invertible sheaf and trivialization are
\[\beta_\nu:\F\times^T\C^{\check{\nu}}\cong \calO_X(\langle\lambda,\check{\nu}\rangle x)\]
we could also attach for any closed point $x\in X$, the group
scheme $G_{\calO,x}\cong G_\calO$ as the group of trivializations
of the trivial $G$-torsor on the disc $\spec\calO_x$, and the
corresponding $G_{\calO,x}$-invariant subvarieties
$\Gr_{G,x}^\lambda\subset\overline{\Gr}_{G,x}^\lambda\subset
\Gr_{G,x}$.
\end{s}

\begin{s}\label{BD Grass}
If we allow the point $x$ to vary in previous construction, we
obtain a global version of the affine Grassmannian. Indeed, we can
construct much more sophisticated geometrical objects thanks to
the moduli interpretation of the affine Grassmannian. This is the
so-called Beilinson-Drinfeld Grassmannian (cf.
\cite{BD2},\cite{MV}). Denote the $n$ fold product of $X$ by
$X^n=X\times\cdots\times X$, and consider the functor
\[\Gr_{G,X^n}(R)=\left\{\begin{array}{l}(x_1,\ldots,x_n)\in X^n(R), \F \mbox{ a } G\mbox{-torsor on } X_R, \\ \beta_{(x_1,\ldots,x_n)} \mbox{ a trivialization of } \F \mbox{ on } X_R-\cup x_i\end{array}\right\}\]
Here we think of the points $x_i: \spec(R)\to X$ as subschemes of
$X_R$ by taking their graphs. This functor is represented by an
ind-scheme which is formally smooth over $X^n$. The one relevant
to our case is $\Gr_{G,X^2}$. Then, for a closed point $(x,y)\in
X^2$,
\[(\Gr_{G,X^2})_{(x,y)}\cong\left\{\begin{array}{ll} \Gr_{G,x}\times \Gr_{G,y} & \mbox{ if }x\neq y\\
                                                    \Gr_{G,x} & \mbox{ if }x=y\end{array}\right.\]

There is also a global version of $G_\calO$, which is
\[G_{\calO,X^n}(R)=\left\{\begin{array}{l}(x_1,\ldots,x_n)\in X^n(R), \gamma_{(x_1,\ldots,x_n)} \mbox{ a trivialization } \\ \mbox{ of the trivial } G\mbox{-torsor } \F_0 \mbox{ on } \widehat{\cup_i x_i}\end{array}\right\}\]
where $\widehat{\cup_i x_i}$ is the formal completion of $X_R$
along $x_1\cup\cdots\cup x_n$. $G_{\calO,X^n}$ is represented by a
formally smooth group scheme over $X^n$. Like $\Gr_{G,X^2}$, if
$x\neq y$, $(G_{\calO,X^2})_{(x,y)}\cong G_{\calO,x}\times
G_{\calO,y}$, and $(G_{\calO,X^2})_{(x,x)}\cong G_{\calO,x}$.
Furthermore, $G_{\calO,X^n}$ acts on $\Gr_{G,X^n}$ in the
following way. Let $(\F,\beta)\in \Gr_{G,X^n}$ and $\gamma\in
G_{\calO,X^2}$, we construct $(\F',\beta)=\gamma\cdot(\F,\beta)$
as follows. If $\F$ can be trivialized along $\widehat{\cup_i
x_i}$, then choose a trivialization and one obtains a transition
function on $\widehat{\cup_i x_i}-\cup_i x_i$. Modifying this
transition function by $\gamma$ and gluing $\F|_{X_R-\cup x_i}$
and $\F_0|_{\widehat{\cup_i x_i}}$ by this new transition
function, one obtains $(\F',\beta)$. If $\F$ cannot be trivialized
along $\widehat{\cup_i x_i}$, choose some faithfully flat $R\to
R'$, apply the same procedure and then use descent.
\end{s}

\begin{s}\label{global orbits}The $G_{\calO,X}$-orbits on $\Gr_{G,X}$ are parameterized
by $\Lambda_G^+$. This is just a direct global counterpart of
\ref{orbits of Gr}. We will denote the orbit corresponding to
$\lambda\in\Lambda_G^+$ by $\Gr_{G,X}^\lambda$ and its closure by
$\overline{\Gr}_{G,X}^\lambda$. Observe that there is a natural
action of $G_{\calO,X}$ on $\overline{\Gr}_{G,X}^\lambda$.
\end{s}

\begin{s}\label{g.o.}Next, we turn to the $G_{\calO,X^2}$-orbits on $\Gr_{G,X^2}$. They are parameterized by
$(\lambda,\mu)\in\Lambda_G^+\times\Lambda_G^+$. Choose a maximal
torus $T$ of $G$. For any
$(\lambda,\mu)\in\Lambda_G\times\Lambda_G$, let
$s^{\lambda,\mu}\in \Gr_{G,X^2}(X^2)$ be $(\pr_1,\pr_2,\F,\beta)$,
where $\pr_1,\pr_2$ are two projections of $X^2$ to $X$,
$(\F,\beta)$ the $T$-torsor over $X^3$ with the trivialization
$\beta$ over $X^3-\Delta_{13}\cup\Delta_{23}$,
$\Delta_{13}=\{(x,y,x)\in X^3\}$ (resp. $\Delta_{23}=\{(y,x,x)\in
X^3\}$) being the graph of $\pr_1$ (resp. $\pr_2$), such that for
any $\check{\nu}\in\check{\Lambda}_G$, the induced invertible
sheaf are
\[\beta_\nu:\F\times^T\C^{\check{\nu}}\cong
\calO_{X^3}(\langle\lambda,\check{\nu}\rangle\Delta_{13}+\langle\mu,\check{\nu}\rangle\Delta_{23})\]
Then $s^{\lambda,\mu}$ is a section of $\Gr_{G,X^2}\to X^2$.
Denote $\Gr_{G,X^2}^{\lambda,\mu}=G_{X^2,\calO}s^{\lambda,\mu}$.
Then $\Gr_{G,X^2}(\C)$ is the union of
$\Gr_{G,X^2}^{\lambda,\mu}(\C)$ for
$(\lambda,\mu)\in\Lambda_G^+\times\Lambda_G^+$. Let
$\overline{\Gr}_{G,X^2}^{\lambda,\mu}$ be the closure of
$\Gr_{G,X^2}^{\lambda,\mu}$. It is a folklore that
$\overline{\Gr}_{G,X^2}^{\lambda,\mu}$ gives a flat degeneration
of $\overline{\Gr}_G^\lambda\times\overline{\Gr}_G^\mu$ to
$\overline{\Gr}_G^{\lambda+\mu}$. As far as the author knows, no
written proof is available. We will prove a slightly different
version of this folklore in Proposition \ref{flat family}, using
Theorem \ref{tensor structure}, which is originally due to
\cite{FL}, Theorem 1.
\end{s}

\begin{s}\label{line bundle}
Assume that $G$ is simple and that $X$ is complete in this
subsection. Let $Bun_{G,X}$ be the moduli stack of principal
$G$-bundles on $X$. Then there are natural morphisms $\pi:\Gr_G\to
Bun_{G,X}$ and $\pi_n:\Gr_{G,X^n}\to Bun_{G,X}$ by simply
forgetting the trivializations. First, assume that $G$ is
simply-connected. Then it is known that
$\mbox{Pic}(Bun_{G,X})\cong\Z$ (cf. \cite{BLS}), and one of the
generators $\calL$ is ample. It is known that $\pi^*\calL$ on
$\Gr_G$ is just $\calL_G$ defined in \ref{basic fact}. Therefore,
one also denotes the invertible sheaf $\pi_n^*\calL$ on
$\Gr_{G,X^n}$ by $\calL_G$. It is clear from the definition that
over $\Gr_{G,X^2}$
\[\calL_G\otimes\C_{(x,y)}\cong\left\{\begin{array}{ll} \calL_G\boxtimes\calL_G  \mbox{ on } \Gr_{G,x}\times \Gr_{G,y} & \mbox{ if }x\neq y\\
                                                    \calL_G \mbox{ on } \Gr_{G,x} & \mbox{ if }x=y\end{array}\right.\]
where $\C_{(x,y)}$ is the skyscraper sheaf at $(x,y)$. Now if $G$
is not simply-connected, we have:
$\pi_0(\Gr_{G,X^2})\cong\pi_1(G)$, and different connected
components of $\Gr_{G,X^2}$ are isomorphic as ind-schemes. Denote
$\calL_G$ the invertible sheaf on $\Gr_{G,X^2}$, the restrictions
of which to each component are all isomorphic to the one on the
neutral component, which is in turn isomorphic to
$\calL_{\tilde{G}}$, where $\tilde{G}$ is the simply-connected
cover of $G$.
\end{s}

\subsection{A flat degeneration of Schubert varieties}\label{flat deg}
\begin{s}
Let $k$ be a positive integer. Then
$H^0(\overline{\Gr}^\lambda_G,\calL_G^{\otimes k})^*$ is a
$G_\calO$-module (and therefore a $G$-module), called the affine
Demazure module (of level $k$). The first main theorem of
\cite{FL} (Theorem 1) claims
\end{s}
\begin{thm}\label{tensor structure}
$H^0(\overline{\Gr}_G^{\lambda+\mu},\calL_G^{\otimes k})^*\cong
H^0(\overline{\Gr}_G^\lambda,\calL_G^{\otimes k})^*\otimes
H^0(\overline{\Gr}_G^\mu,\calL_G^{\otimes k})^*$ as $G$-modules.
\end{thm}
The proof of this theorem in \emph{loc. cit.} is of combinatoric
flavor. Here we give a purely algebro-geometrical proof.
\begin{proof}Fix a closed point $p\in X$. We consider the following ind-scheme over $X$.
\[\Gr_{G,X}\tilde{\times}\Gr_{G,p}(R)=\left\{\begin{array}{l}x\in X(R), \F_1, \F_2 \mbox{ two } G\mbox{-torsors on } X_R,\\ \beta_1 \mbox{ a trivialization of } \F_1 \mbox{ on } X_R-x, \beta_2:\F_2|_{(X-p)_R}\cong\F_1|_{(X-p)_R}\end{array}\right\}\]
We have the projection
\[p:\Gr_{G,X}\tilde{\times}\Gr_{G,p}\to \Gr_{G,X}\]
sending $(x,\F_1,\F_2,\beta_1,\beta_2)$ to $(x,\F_1,\beta_1)$.
This map realizes $\Gr_{G,X}\tilde{\times}\Gr_{G,p}$ as a
fibration over $\Gr_X$, with fibers isomorphic to $\Gr_{G,p}$.

Indeed, there is a $G_{\calO,p}$-torsor $\mathcal P$ over
$\Gr_{G,X}$ given by
\[\mathcal P(R)=\left\{\begin{array}{l}x\in X(R), \F \mbox{ a } G\mbox{-torsora on } X_R,\\ \beta_1 \mbox{ a trivialization of } \F \mbox{ on } X_R-x, \beta_2\mbox{ a trivialization of } \F \mbox{ on } \hat{p}\end{array}\right\}\]
Then it is easy to see that $\mathcal P\times^{G_{\calO,p}}
\Gr_{G,p}\cong \Gr_{G,X}\tilde{\times}\Gr_{G,p}$. Observe that
$\mathcal P|_{X-p}$ is a indeed a trivial $G_{\calO,p}$-torsor,
with the section $\Gr_{G,X}|_{X-p}\to\mathcal P|_{X-p}$ given by
$(x,\F,\beta)\mapsto(x,\F,\beta,\beta|_{\hat{p}})$. Therefore,
\[\Gr_{G,X}\tilde{\times}\Gr_{G,p}|_{X-p}\cong \Gr_{G,X-p}\times \Gr_{G,p}\]
On the other hand, it is clear
\[\Gr_{G,p}\tilde{\times}\Gr_{G,p}:=\Gr_{G,X}\tilde{\times}\Gr_{G,p}|_p\cong G_{\K_p}\times^{G_{\calO_p}}\Gr_{G,p}\]

Now we denote
\[\overline{\Gr}_{G,X}^\lambda\tilde{\times}\overline{\Gr}_{G,p}^\mu=\mathcal P|_{\overline{\Gr}_{G,X}^\lambda}\times^{G_{\calO,p}}\overline{\Gr}_{G,p}^\mu\]
This is a fibration over $\overline{\Gr}_{G,X}^\lambda$ with
fibers isomorphic to $\overline{\Gr}_{G,p}^\mu$. Furthermore,
\[\overline{\Gr}_{G,X}^\lambda\tilde{\times}\overline{\Gr}_{G,p}^\mu|_{X-p}\cong\overline{\Gr}_{G,X-p}^\lambda\times\overline{\Gr}_{G,p}^\mu\]
Since $\overline{\Gr}_{G,X}^\lambda$ is a fibration over $X$ with
fibers isomorphic to $\overline{\Gr}_{G}^\lambda$, we obtain that
$\overline{\Gr}_{G,X}^\lambda\tilde{\times}\overline{\Gr}_{G,p}^\mu$
is flat over $X$, and that the special fiber
\[\overline{\Gr}_{G,p}^\lambda\tilde{\times}\overline{\Gr}_{G,p}^\mu:=\overline{\Gr}_{G,X}^\lambda\tilde{\times}\overline{\Gr}_{G,p}^\mu|_p\]
is a fibration over $\overline{\Gr}_{G,p}^\lambda$ with fibers
isomorphic to $\overline{\Gr}_{G,p}^\mu$. In particular, the
special fiber is reduced.

Observe that we have the natural map
\[m:\Gr_{G,X}\tilde{\times}\Gr_{G,p}\to \Gr_{G,X^2}|_{X\times p}\]
by sending $(x,\F_1,\F_2,\beta_1,\beta_2)$ to
$(x,p,\F_2,\beta_1\circ\beta_2)$. This is an isomorphism away from
$p$ and over $p$, it is just the usual convolution
\[m_p:G_{\K_p}\times^{G_{\calO_p}}\Gr_{G,p}\to \Gr_{G,p}\]
Observe that over the special fiber
\[m_p:\overline{\Gr}_{G,p}^\lambda\tilde{\times}\overline{\Gr}_{G,p}^\mu\to\overline{\Gr}_{G,p}^{\lambda+\mu}\]
is indeed a partial Bott-Samelson resolution.

Recall the invertible sheaf $\calL_G^{\otimes k}$ on
$\Gr_{G,X^2}$. Then by \ref{line bundle}, the restriction of
$m^*\calL_G^{\otimes k}$ to
$\Gr_{G,X}\tilde{\times}\Gr_{G,p}|_{X-p}\cong \Gr_{G,X-p}\times
\Gr_{G,p}$ is just $\calL_G^{\otimes k}\boxtimes\calL_G^{\otimes
k}$. Since
$\overline{\Gr}_{G,X}^\lambda\tilde{\times}\overline{\Gr}_{G,p}^\mu$
is a flat family over $X$, we obtains
\[H^0(\overline{\Gr}_{G,x}^\lambda,\calL_G^{\otimes k})\otimes H^0(\overline{\Gr}_{G,p}^\mu,\calL_G^{\otimes k})\cong H^0(\overline{\Gr}_{G,p}^\lambda\tilde{\times}\overline{\Gr}_{G,p}^\mu,m_p^*\calL_G^{\otimes k})\]
Finally, the theorem follows from the well-known that
$H^0(\overline{\Gr}_{G,p}^\lambda\tilde{\times}\overline{\Gr}_{G,p}^\mu,m_p^*\calL_G^{\otimes
k})\cong H^0(\overline{\Gr}_{G,p}^{\lambda+\mu},\calL_G^{\otimes
k})$ (e.g. \cite{Ku} Theorem 2.16).
\end{proof}

\begin{s}\label{f.f.}We still fix a point $p\in X$. Recall the variety
$\overline{\Gr}_{G,X^2}^{\lambda,\mu}$ from \ref{g.o.}. Denote the
reduced base change scheme by $\overline{\Gr}_{G,X\times
p}^{\lambda,\mu}:=(\overline{\Gr}_{G,X^2}^{\lambda,\mu}|_{X\times
p})_{red}$. We prove that
\end{s}

\begin{prop}\label{flat family} $\overline{\Gr}_{G,X\times p}^{\lambda,\mu}$ is a scheme flat over $X\cong X\times p$, whose
fiber over $x\neq p$ is
$\overline{\Gr}_{G,x}^\lambda\times\overline{\Gr}_{G,p}^\mu$ and
whose fiber over $p$ is $\overline{\Gr}_{G,p}^{\lambda+\mu}$.
\end{prop}
\begin{proof} We give another interpretation of
$\overline{\Gr}_{G,X\times p}^{\lambda,\mu}$. Recall the section
$s^{\lambda,\mu}$ of $\Gr_{G,X^2}$ from \ref{g.o.}. The
restriction $s^{\lambda,\mu}|_{X\times p}$ gives a section of
$\Gr_{G,X^2}|_{X\times p}$ over $X\times p$, also denoted by
$s^{\lambda,\mu}$. The group scheme $\mathcal
G:=G_{\calO,X^2}|_{X\times p}$ acts on $\Gr_{G,X^2}|_{X\times p}$.
Then $\overline{\Gr}_{G,X\times p}^{\lambda,\mu}$ is the closure
of the orbit $\mathcal G s^{\lambda,\mu}$ in
$\Gr_{G,X^2}|_{X\times p}$. In other words, $s^{\lambda,\mu}$
defines a morphism $\mathcal G\to \Gr_{G,X^2}|_{X\times p}$ by
\[\varphi:\mathcal G=\mathcal G\times_X X\stackrel{id\times s^{\lambda,\mu}}{\longrightarrow}\mathcal G\times_X \Gr_{G,X^2}|_{X\times p}\longrightarrow \Gr_{G,X^2}|_{X\times p}\]
Then $\overline{\Gr}_{G,X\times p}^{\lambda,\mu}$ is the
scheme-theoretic image of this morphism. That is,
$\calO_{\overline{\Gr}_{G,X\times p}^{\lambda,\mu}}$ is the image
sheaf of the comorphism $\calO_{\Gr_{G,X^2}|_{X\times
p}}\to\varphi_*\calO_{\mathcal G}$, and therefore, a subsheaf of
$\varphi_*\calO_{\mathcal G}$. Since $\mathcal G$ is formally
smooth over $X$, the local parameter of $X$ is not a zero divisors
in $\varphi_*\calO_{\mathcal G}$, nor in
$\calO_{\overline{\Gr}_{G,X\times p}^{\lambda,\mu}}$. This proves
that $\overline{\Gr}_{G,X\times p}^{\lambda,\mu}$ is flat over
$X$.

It is easy to see that
\[((\overline{\Gr}_{G,X\times p}^{\lambda,\mu})_x)_{red}=\left\{\begin{array}{ll}\overline{\Gr}_{G,x}^\lambda\times\overline{\Gr}_{G,p}^\mu &x\neq p\\
\overline{\Gr}_{G,p}^{\lambda+\mu} &x=p\end{array}\right.\] Away
from $p$, $\overline{\Gr}_{G,X\times p}^{\lambda,\mu}$ is
\'{e}tale locally trivial over $X-p$, and therefore,
$(\overline{\Gr}_{G,X\times
p}^{\lambda,\mu})_x=\overline{\Gr}_{G,x}^\lambda\times\overline{\Gr}_{G,p}^\mu$
for $x\neq p$. (Or if one assumes that $X=\A^1$, then the family
is in fact trivial.) Now $\overline{\Gr}_{G,p}^{\lambda+\mu}$ is
the closed subscheme of $(\overline{\Gr}_{G,X\times
p}^{\lambda,\mu})_p$ defined by the nilpotent radical. We have the
exact sequence of sheaves
\[0\to\mathcal J\to \calO_{(\overline{\Gr}_{G,X\times p}^{\lambda,\mu})_p}\to\calO_{\overline{\Gr}_{G,p}^{\lambda+\mu}} \to 0\]
Tensoring the ample invertible sheaf $\calL_G^{\otimes k}$
constructed in \ref{line bundle}, where $k$ is sufficiently large,
one obtains
\[0\to H^0((\overline{\Gr}_{G,X\times p}^{\lambda,\mu})_p,\mathcal J\otimes\calL_G^{\otimes k})\to H^0((\overline{\Gr}_{G,X\times p}^{\lambda,\mu})_p,\calL_G^{\otimes k})\to H^0(\overline{\Gr}_{G,p}^{\lambda+\mu},\calL_G^{\otimes k}) \to 0\]
By the flatness of $\overline{\Gr}_{G,X\times p}^{\lambda,\mu}$
and Theorem \ref{tensor structure},
$H^0((\overline{\Gr}_{G,X\times
p}^{\lambda,\mu})_p,\calL_G^{\otimes k})$ and
$H^0(\overline{\Gr}_{G,p}^{\lambda+\mu},\calL_G^{\otimes k})$ are
both isomorphic to $H^0(\overline{\Gr}_G^\lambda,\calL_G^{\otimes
k})\otimes H^0(\overline{\Gr}_G^\mu,\calL_G^{\otimes k})$. Whence,
$H^0((\overline{\Gr}_{G,X\times p}^{\lambda,\mu})_p,\mathcal
J\otimes\calL_G^{\otimes k})=0$. This implies $\mathcal J=0$ since
$\mathcal J\otimes\calL_G^{\otimes k}$ are generated by global
sections.
\end{proof}

\subsection{$T$-fixed point subscheme of the affine
Grassmannian}\label{fp}
\begin{s}
For any group scheme $G$ over $\C$ operating on a presheaf $Y$ on
$\mathbf{Aff}_\C$, there is the notion of the fixed point
subfunctor $Y^G\subset Y$ as defined in \cite{DG}
\[Y^G(R):=\{x\in Y(R)| \mbox { for all } R\to R',\mbox{ and  } g\in G(R'),
gx_{R'}=x_{R'} \}\] where $x_{R'}$ denote the image of $x$ under
the natural restriction $Y(R)\to Y(R')$. It is easy to check that
if $Y$ is a sheaf, then $Y^G$ is a subsheaf. Something that is
less trivial (cf. \cite{DG}, II, \S 1, Theorem 3.6, (d)) is the
following
\end{s}

\begin{lem}
If $Y$ is a separated $\C$-scheme, then $Y^G$ is a closed
subscheme of $Y$.
\end{lem}

\begin{s}
Now assume that $G$ is a reductive algebraic group over $\C$. Let
$\iota:L\to G$ be a Levi subgroup of some parabolic subgroup of
$G$. Denote also by $\iota:\Gr_L\to \Gr_G$ the embedding induced
from $\iota:L\to G$. Denote $T=Z(L)^0$ the neutral connected
component of the center $Z(L)$ of $L$. Then $T$ is a torus of $G$,
and therefore acts on $\Gr_G$. It is well-known that the
$T(\C)$-fixed point set of $\Gr_G(\C)$ is just $\Gr_L(\C)$. The
following theorem claims that this even holds
scheme-theoretically.
\end{s}

\begin{thm}\label{Fixed-point}
The morphism $\iota: \Gr_L\to \Gr_G$ identifies $\Gr_L$ as the
$T$-fixed point subsheaf of $\Gr_G$.
\end{thm}
\begin{proof} It is obvious that $\Gr_L$ is contained in
$(\Gr_G)^T$. So we prove the converse. Observe that $\Gr_G$ is can
be also obtained as the sheafification of the presheaf $R\mapsto
G(R\hat{\otimes}\K)/G(R\hat{\otimes}\calO)$ in the \'{e}tale
topology, it is enough to show for any $R$ strictly Henselian
local ring, $\Gr_L(R)\to (\Gr_G)^T(R)$ is surjective. We need the
following lemma, whose proof is communicated to the author by
Zhiwei Yun.

\begin{lem}
Let $R$ be a $\C$-algebra. If $\mathrm{Pic}(\mathrm{Spec} R)=0$,
then
\[G(R\hat{\otimes}\K)=B(R\hat{\otimes}\K)G(R\hat{\otimes}\calO)\]
\end{lem}

\begin{proof} Let $X=G/B$ be the flag variety of $G$ defined over
$\C$. Then for any $\C$-algebra $A$, there is an injective map of
sets $G(A)/B(A)\hookrightarrow X(A)$, since $G(A)/B(A)$ is the set
of isomorphism classes of pairs $(\F_0,\beta)$, where $\F_0$ is a
trivial $B$-torsor over $\spec A$ and $\beta: \F_0\to G$ is a
$B$-equivariant morphism, whereas $X(A)$ is the set of isomorphism
classes of pairs $(\F,\beta)$, where $\F$ a $B$-torsor over $\spec
A$ and $\beta: \F\to G$ is a $B$-equivariant morphism.

Now let $F=R\hat{\otimes}\K, S=R\hat{\otimes}\calO$. Then one has
$X(F)=X(S)$ because $X$ is projective and one can always make an
$S$-point out of an $F$-point by multiplying a common denominator
of coordinates. Therefore, the lemma holds if $G(S)/B(S)=X(S)$.
Clearly from previous description, it is enough to show that
$H^1(\spec S,B)=0$. Observe that $B$ can be filtered by normal
subgroups such that the associated quotients are $\G_m$s and
$\G_a$s. Since $\G_a$ is a coherent sheaf, $H^1(\spec S,\G_a)$ is
always trivial. To show $H^1(\spec S,\G_m)=0$, observe that any
$\G_m$-torsor $\F$ over $\spec S$ has a section when restricted to
$\spec R\hookrightarrow\spec S$ since $\mbox{Pic}(\spec R)=0$. Now
$\F$ is smooth over $\spec S$, so the section of $\F|_{\spec R}$
extends to a section of $\F$ over $\spec S$. Therefore, the lemma
follows.
\end{proof}

\noindent\emph{Conclusion of the proof of the theorem.} Fix a
uniformizer $t\in\calO$ for convenience. Let $P$ be a parabolic
subgroup of $G$ containing $L$. By above lemma
$G(R((t)))=P(R((t)))G(R[[t]])$ since $R$ is strictly Henselian.
Now let $x\in (\Gr_G)^T(R)$, one could assume that $x$ is
represented by an element $g\in P(R((t)))$. Then $x$ being a
$T$-fixed point is equivalent to $g^{-1}tg\in P(R[[t]])$ for any
$t\in T(R)$. Let $U_P$ be unipotent radical of $P$. Since $L\times
U_P\to P$ is an isomorphism (as $\C$-varieties), one could write
$g=g'g''$ with $g'\in L(R((t))), g''\in U_P(R((t)))$. Furthermore,
since $U_P$ is unipotent, $g''$ could be uniquely written as
$g''=g_1g_2$ with $g_1\in U_P(t^{-1}R[t^{-1}])$ and $g_2\in
U_P(R[[t]])$. Then $g^{-1}tg=g_2^{-1}g_1^{-1}tg_1g_2\in P(R[[t]])$
is equivalent $g_1=tg_1t^{-1}$ for any $t\in T(R)$. Therefore
$g_1=1$, and $x$ is represented by an element $g\in
L(R((t)))G(R[[t]])$, that is $x\in \Gr_L(R)$.
\end{proof}

\begin{cor}\label{T-fix point}
Let $\iota:T\to G$ be a maximal torus of $G$. Then $\iota:\Gr_T\to
\Gr_G$ identifies $\Gr_T$ as the $T$-fixed point subsheaf of
$\Gr_G$.
\end{cor}

\begin{s}
Let $M=[L,L]$ be the derived group of a Levi subgroup $L$ of $G$.
Then $M$ is a semisimple subgroup of $G$. Let $T_M\subset T_G$ be
the maximal torus of $M$ and $G$. For any $s^\lambda\in \Gr_G$,
the morphism $M_\K\to \Gr_G$ given by $m\to m\cdot s^\lambda$
induces a closed embedding of $\Gr_M\to \Gr_G$ by \cite{MOV} Lemma
3.2. It is easy to see that $\Gr_M$ is $T_G$-invariant. We have
another corollary of Theorem \ref{Fixed-point}.
\end{s}

\begin{cor}\label{fixed-point in Levi}
Under the above assumption, $(\Gr_M)^{T_G}\cong (\Gr_M)^{T_M}$.
\end{cor}

\begin{proof} Recall the natural embedding $\Gr_L\subset \Gr_G$. For simplicity, we assume that $G$ is
simply-connected. (The general case is very similar.) Then $M$ is
simply-connected and therefore, $\Lambda_M=R_M$. The connected
components of $\Gr_L$ are labelled by elements in
$\Lambda_G/\Lambda_M$. Then the way of embedding of $\Gr_M$ into
$\Gr_G$ realized $\Gr_M$ as the reduced subscheme of the component
of $\Gr_L$ corresponding to $\lambda \mbox{ mod } \Lambda_M$.
Observe that $T_G=T_M\cdot Z(L)^0$, where $Z(L)^0$ is the neutral
component of the center of $L$. Therefore, $(\Gr_M)^{T_G}\cong
((\Gr_M)^{Z(L)^0})^{T_M}=(\Gr_M)^{T_M}$.
\end{proof}

\subsection{The Borel-Weil theorem}\label{BW} Assume
that $G$ is a simple algebraic group in this section. Let $\frakg$
be its Lie algebra, and $\hat{\frakg}$ be the untwisted affine
algebra associated to $\frakg$.

\begin{s}
Let $\Gr_G$ be the affine Grassmannian of $G$, and $\calL_G$ an
invertible sheaf on $\Gr_G$, which is the positive generator of
the Picard group of each connected component. (See Theorem 1.1.2
(3) and (4).) It is well-known that $\Gamma(\Gr_G,\calL_G)$ has a
natural $\hat{\frakg}$-module structure.
\end{s}

\begin{s}
Recall that the irreducible integrable representations of
$\hat{\frakg}$ are parameterized by $(k,\check{\nu})$, where $k$
is a positive integer, called the level, and $\check{\nu}$ is a
dominant integral weight of $\frakg$ such that
$(\check{\theta},\check{\nu})^*\leq k$. For such
$(k,\check{\nu})$, the irreducible integrable representation of
highest weight $(k\Lambda+\check{\nu})$ is denoted by
$L(k\Lambda+\check{\nu})$.
\end{s}

\begin{s}\label{minuscule coweight}
Recall that a fundamental coweight $\omega_i$ of $\frakg$ is
called minuscule if $\langle\omega_i,\check{\alpha}\rangle\leq 1$
for any $\check{\alpha}$ positive root of $\frakg$. We will also
include the zero coweight as a minuscule coweight. It can be shown
(for example, see \cite{Bour} exercises in Chap. VI) that for any
$\gamma\in \pi_1(G)\cong\Lambda_G/R_G, \gamma\neq 0$, there is a
unique $i_\gamma\in I$, such that: (i)
$\omega_{i_\gamma}\in\Lambda_G$; (ii)$\omega_{i_\gamma} \mbox{ mod
} R_G=\gamma$; (iii) $\omega_{i_\gamma}$ is a minuscule coweight.
If $\gamma=0$, we set the corresponding minuscule coweight
$\omega_{i_0}=0$.
\end{s}

We have the following Borel-Weil theorem

\begin{prop}\label{Borel-Weil}
Assumptions are as above. One has
\[\Gamma(\Gr_G,\calL_G^{\otimes k})^*\cong \bigoplus_{\gamma\in\pi_1(G)}L(k\Lambda+k\iota\omega_{i_\gamma})\]
as $\hat{\frakg}$-modules.
\end{prop}

\begin{proof}Choose $T\subset G$
a maximal torus. Let $\tilde{G}\to G$ be the simply-connected
cover of $G$. Then $\pi_1(G)\cong \Lambda_G/R_G$ is also regarded
as a subgroup of $\tilde{G}$. For any $\gamma\in\pi_1(G)$, choose
$\lambda_\gamma\in\Lambda_G$ a lifting of $\gamma$, e.g.
$\lambda_\gamma=\omega_{i_\gamma}$. One also fixes a uniformizer
$t\in\calO$. From the exact sequence of \'{e}tale sheaves on
$\spec\K$
\[1\to\pi_1(G)\to\tilde{G}\to G\to 1\]
one obtains that
\[0\to \pi_1(G)\to\tilde{G}_\K\to G_\K\to H^1(\spec\K,\pi_1(G))\cong\pi_1(G)\to 0\]
and therefore
\[G_\K=\bigsqcup_{\gamma\in\pi_1(G)} t^{\lambda_\gamma}(\tilde{G}_\K/\pi_1(G))\]
Since $t^\lambda\in G(\K)$, $Ad_{t^\lambda}$ acts on
$\tilde{G}_\K/\pi_1(G)=G_\K^0$ by adjoint action. It is liftable
to a unique automorphism of $\tilde{G}_\K$, also denoted by
$Ad_{t^\lambda}$. Now
\[\Gr_G=\bigsqcup_{\gamma\in\pi_1(G)} (\Gr_G)^\gamma=\bigsqcup_{\gamma\in\pi_1(G)} \tilde{G}_\K/Ad_{t^{\lambda_\gamma}}\tilde{G}_\calO\]
Observe
$(\Gr_G)^\gamma=\tilde{G}_\K/Ad_{t^{\lambda_\gamma}}\tilde{G}_\calO$
for different $\gamma$'s are isomorphic as ind-varieties, but not
as $\tilde{G}_\K$-homogeneous spaces. Indeed,
$Ad_{t^{\lambda_\gamma}}\tilde{G}_\calO$ is the hyperspecial
parahoric subgroup in $\tilde{G}_\K$ obtained by deleting the
vertex $i_\gamma$ of the affine Dynkin diagram. Therefore,
$\Gamma((\Gr_G)^\gamma,\calL_G^{\otimes k})^*\cong
L(k\Lambda+k\check{\omega}_{i_\gamma})$ as $\hat{\frakg}$-modules
by \cite{Ku} and \cite{Ma}. Remark that in \cite{Ku} and
\cite{Ma}, the definition of the (partial) flag variety is
different from the definition here. However, the two definitions
are the same due to \cite{BL1} Theorem 7.7. Since
$\check{\alpha}_{i_\gamma}$ is always a long root for $\gamma\neq
0$,
$\iota\alpha_{i_\gamma}=\dfrac{2\check{\alpha}_{i_\gamma}}{(\check{\alpha}_{i_\gamma},\check{\alpha}_{i_\gamma})^*}=\check{\alpha}_{i_\gamma}$.
And therefore, $\check{\omega}_{i_\gamma}=\iota\omega_{i_\gamma}$
\end{proof}

\section{Proof of Theorem \ref{Main Theorem}}\label{main proof}
In this section, we will prove the main theorem. Our strategy is
first to reduce the full theorem to some special cases, where the
geometry of Schubert varieties is simple, and then to prove these
special cases. We will also discuss the smooth locus of Schubert
varieties in this section.
\subsection{First Reductions}\label{first reduction} Let $G$ be a simple algebraic
group, which is not assumed to be simply-laced in this section. We
will continue using the notations introduced in \S 1. Recall that
$(\overline{\Gr}_G^\lambda)^T$ is a finite scheme supported at
$\{s^{w\mu}|\mu\in\Lambda_G^+,\mu\leq\lambda,w\in W\}$. Let
$\I^\lambda\subset\calO_{\overline{\Gr}_G^\lambda}$ be the ideal
sheaf defining $(\overline{\Gr}_G^\lambda)^T$. For any
quasi-coherent sheaf $\F$ on $\Gr_G$, we will denote
$\F(n)=\F\otimes\calL_G^{\otimes n}$. We prove that

\begin{prop}\label{surjectivity}
For $k$ any positive integer, the natural morphism
$\calL_G^{\otimes
k}\to\calO_{(\overline{\Gr}_G^\lambda)^T}\otimes\calL_G^{\otimes
k}$ induces a surjective map
\[\Gamma(\overline{\Gr}_G^\lambda,\calL_G^{\otimes
k})\twoheadrightarrow\Gamma(\overline{\Gr}_G^\lambda,\calO_{(\overline{\Gr}_G^\lambda)^T}\otimes\calL_G^{\otimes
k})\]
\end{prop}
\begin{proof} By Corollary \ref{T-fix point}, the natural embedding $\Gr_T\to \Gr_G$ identifies $\Gr_T$ as the
$T$-fixed subscheme of $\Gr_G$. Therefore,
\[\Gamma(\Gr_T,\calL_G^{\otimes
k})=\lim\limits_{\longleftarrow}\Gamma(\overline{\Gr}_G^\lambda,\calO_{(\overline{\Gr}_G^\lambda)^T}\otimes\calL_G^{\otimes
k})\] where the limit is taken over $\lambda\in\Lambda_G^+$. On
the other hand, we have \[\Gamma(\Gr_G,\calL_G^{\otimes
k})=\lim\limits_{\longleftarrow}\Gamma(\overline{\Gr}_G^\lambda,\calL_G^{\otimes
k})\] Since $\Gamma((\overline{\Gr}_G^\lambda)^T,\calL_G^{\otimes
k})\to\Gamma((\overline{\Gr}_G^\mu)^T,\calL_G^{\otimes k})$ is
surjective if $\lambda\geq\mu$, to prove the proposition, it is
enough to show that the natural map
\[\Gamma(\Gr_G,\calL_G^{\otimes k})\to\Gamma(\Gr_T,\calL_G^{\otimes
k})\] is surjective, or equivalently,
\[\Gamma(\Gr_T,\calL_G^{\otimes
k})^*\to\Gamma(\Gr_G,\calL_G^{\otimes k})^*\] is injective.

Observe that both $\Gamma(\Gr_T,\calL_G^{\otimes k})^*$ and
$\Gamma(\Gr_G,\calL_G^{\otimes k})^*$ have module structures over
the Heisenberg Lie algebra $\hat{\frakt}\subset\hat{\frakg}$,
which is the restriction of the central extension of
$\frakg\hat{\otimes}\K$ to $\frakt\hat{\otimes}\K$ (see
\ref{Heisenberg}). Furthermore, the map
$\Gamma(\Gr_T,\calL_G^{\otimes
k})^*\to\Gamma(\Gr_G,\calL_G^{\otimes k})^*$ is a
$\hat{\frakt}$-mod morphism. Since $\Gr_T$ is discrete,
\[\Gamma(\Gr_T,\calL_G^{\otimes
k})=\bigoplus_{\lambda\in\Lambda_G}\calO_{\Gr_T,s^\lambda}\otimes\calL_G^{\otimes
k}|_{s^\lambda}\] and \ref{formal part} implies that as
$\hat{\frakt}$-modules,
$(\calO_{\Gr_T,s^\lambda}\otimes\calL_G^{\otimes
k}|_{s^\lambda})^*\cong\pi^k_{-k\lambda}$ is the Fock
$\hat{\frakt}$-module as defined in \ref{Heisenberg}.

To prove the proposition, it is enough to show that the map
$\pi^k_{-k\lambda}\to\Gamma(\Gr_G,\calL_G^{\otimes k})^*$ is not
zero. Then it must be injective since $\pi^k_{-k\lambda}$ is
simple. Dually, it is enough to show
\[\Gamma(\Gr_G,\calL_G^{\otimes
k})\to\calO_{\Gr_T,s^\lambda}\otimes\calL_G^{\otimes
k}|_{s^\lambda}\] is not zero. However, this is clear. For any
$\lambda$, let $\sigma_\lambda$ be the linear form on
$\bigoplus_{\gamma\in\pi_1(G)}L(\Lambda+\iota\omega_{i_\gamma})$,
which is not zero along $t^\lambda v_\Lambda$, and zero on any
other weight vectors, where $t$ is some chosen uniformizer, and
$v_\Lambda$ is the line of highest weight vectors in $L(\Lambda)$.
(See Lemma \ref{ration curves} for the meaning of $t^\lambda
v_\Lambda$.) Then
$\sigma_\lambda^k\in\Gamma(\Gr_G,\calL_G^{\otimes k})$ maps
nonzero to $\calO_{\Gr_T,s^\lambda}\otimes\calL_G^{\otimes
k}|_{s^\lambda}$.
\end{proof}

Therefore, Theorem \ref{Main Theorem} follows from the following
\begin{thm}\label{reduced main theorem }
Assume that $G$ is a simple algebraic group of type $A$ or $D$.
Then for any $\lambda\in\Lambda_G$,
$\Gamma(\overline{\Gr}_G^\lambda,\I^\lambda(1))=0$.
\end{thm}

This theorem is also proved for certain $\lambda$ if $G$ is of
type $E$ (see \S \ref{type E,I}-\ref{type E,II}). We expect it
holds for all $\lambda$ in this case.

We next show that, to prove Theorem \ref{reduced main theorem },
it is enough to prove it for fundamental coweights.
\begin{prop}\label{factorization}
Let $G$ be a simple, but not necessarily simply-laced algebraic
group. For $\lambda,\mu\in\Lambda_G^+$, if
$\Gamma(\overline{\Gr}_G^\lambda,\I^\lambda(1))=\Gamma(\overline{\Gr}_G^\mu,\I^\mu(1))=0$,
then
$\Gamma(\overline{\Gr}_G^{\lambda+\mu},\I^{\lambda+\mu}(1))=0$.
\end{prop}
\begin{proof} Recall the variety
$\overline{\Gr}_{G,X\times p}^{\lambda,\mu}$ from \ref{f.f.}. $T$
acts on $\overline{\Gr}_{G,X\times p}^{\lambda,\mu}$ by embedding
$T_{X\times p}\hookrightarrow \mathcal G=G_{\calO,X^2}|_{X\times
p}$. Let $(\overline{\Gr}_{G,X\times p}^{\lambda,\mu})^T$ be the
$T$-fixed subscheme of $\overline{\Gr}_{G,X\times
p}^{\lambda,\mu}$. Then its fiber over $x\neq p$ is
$(\overline{\Gr}_{G,x}^\lambda)^T\times(\overline{\Gr}_{G,p}^\mu)^T$
and over $p$ is $(\overline{\Gr}_{G,p}^{\lambda+\mu})^T$.

For simplicity, we will assume that $X=\A^1$ now. Then we claim
that $(\overline{\Gr}_{G,X\times p}^{\lambda,\mu})^T$ is flat over
$X-p$. To see this, one regards $X-p\cong \G_m$ so that $\G_m$
acts on $X-p$. It is easy to see there is a $\G_m$-action on
$\overline{\Gr}_{G,X\times p}^{\lambda,\mu}|_{X-p}$, commuting
with the $T$-action, such that the natural projection
$\overline{\Gr}_{G,X\times p}^{\lambda,\mu}|_{X-p}\to X-p$ is an
equivariant map. Therefore, $\overline{\Gr}_{G,X\times
p}^{\lambda,\mu}|_{X-p}\cong
(X-p)\times(\overline{\Gr}_G^\lambda\times\overline{\Gr}_G^\mu)$
and
$(\overline{\Gr}_G^{\lambda,\mu})^T|_{X-p}\cong(X-p)\times((\overline{\Gr}_G^\lambda)^T\times(\overline{\Gr}_G^\mu)^T)$.
Our claim follows. Now let $Z$ be the scheme theoretical image of
$(\overline{\Gr}_{G,X\times
p}^{\lambda,\mu})^T|_{X-p}\hookrightarrow\overline{\Gr}_{G,X\times
p}^{\lambda,\mu}$. Since $\overline{\Gr}_{G,X\times
p}^{\lambda,\mu}$ is flat over $X$ (cf. Proposition \ref{flat
family}), so is $Z$. Certainly,
$Z\subset(\overline{\Gr}_{G,X\times p}^{\lambda,\mu})^T$.
Therefore, $Z_p\subset(\overline{\Gr}_{G,X\times
p}^{\lambda,\mu})^T|_p\cong(\overline{\Gr}_{G,p}^{\lambda+\mu})^T$.
Since both of them are finite schemes, we have
\[\begin{array}{lll}    &\dim\calO_{(\overline{\Gr}_G^{\lambda+\mu})^T}\\
                    \geq&\dim\calO_{Z_p}&\\
                       =&\dim\calO_{(\overline{\Gr}_G^\lambda)^T}\dim\calO_{(\overline{\Gr}_G^\mu)^T}&\mbox{ (by the flatness of } Z)\\
                       =&\dim H^0(\overline{\Gr}_G^\lambda,\calL_G)\dim H^0(\overline{\Gr}_G^\mu,\calL_G)&\mbox{ (by the assumption of the
                                                                 lemma) }\\
                       =&\dim H^0(\overline{\Gr}_G^{\lambda+\mu},\calL_G)&\mbox{ (by Theorem \ref{tensor structure}) }\\
                    \geq&\dim\calO_{(\overline{\Gr}_G^{\lambda+\mu})^T}&\mbox{ (by
                    Proposition
\ref{surjectivity}) }\end{array}\] Therefore, $\dim
H^0(\overline{\Gr}_G^{\lambda+\mu},\calL_G)=\dim\calO_{(\overline{\Gr}_G^{\lambda+\mu})^T}$.
Since
\[H^0(\overline{\Gr}_G^{\lambda+\mu},\calL_G)\twoheadrightarrow H^0(\overline{\Gr}_G^{\lambda+\mu},\calO_{(\overline{\Gr}_G^{\lambda+\mu})^T}\otimes\calL_G)\] is a surjective by Proposition \ref{surjectivity}, and they have the same
dimension, the map must be an isomorphism.
\end{proof}

\subsection{Proof of Theorem \ref{reduced main theorem }}\label{proof}
By Proposition \ref{factorization}, it is enough to prove the
Theorem \ref{reduced main theorem } for $\lambda$ being a
fundamental coweight. Therefore, we could assume from now on that
$G$ is of adjoint type. However, $G$ is not assumed to be
simply-laced through \ref{a}-\ref{addimisible}.

\begin{s}\label{a}Assume that $G$ is a simple algebraic group. For $\check{\psi}=n\check{\delta}+\check{\alpha}$ a real root of
$\hat{\frakg}$, let $U_{\check{\psi}}$ be the corresponding
unipotent subgroup in $G_\K$, i.e., $\mbox{Lie}U_{\check{\psi}}=\C
e_{n\check{\delta}+\check{\alpha}}$, where
$e_{n\check{\delta}+\check{\alpha}}$ is a root vector of
$n\check{\delta}+\check{\alpha}$. Let $S_{\check{\psi}}\subset
G_\K$ be the subgroup generated by $U_{\check{\psi}},
U_{-\check{\psi}}$. For $\lambda\in\Lambda_G$, let $s^\lambda$ be
the corresponding point in $\Gr_G$. For any choice of the
uniformizer $t$, recall $t^\lambda\in G_\K$. Since
\[t^{-\lambda}U_{n\check{\delta}+\check{\alpha}}t^\lambda=U_{(n-\langle\check{\alpha},\lambda\rangle)\check{\delta}+\check{\alpha}}\]
One obtains that $U_{n\check{\delta}+\check{\alpha}}s^\lambda\neq
s^\lambda$ if and only if
$n<\langle\check{\alpha},\lambda\rangle$.
\end{s}
\begin{lem}\label{ration curves}
Under above assumptions, $S_{\check{\psi}}\cdot s^\lambda\subset
\Gr_G$ is a $T$-invariant rational curve on $\Gr_G$. There are two
$T$-fixed points in $S_{\check{\psi}}\cdot s^\lambda$; namely,
$s^\lambda$ and
$s^{\lambda-(\langle\lambda,\check{\alpha}\rangle-n)\alpha}$.
Furthermore, the restriction of $\calL_G$ to this rational curve
is
$\calO(\dfrac{2(\langle\lambda,\check{\alpha}\rangle-n)}{(\check{\alpha},\check{\alpha})^*})$.

In particular, if $\check{\alpha}$ is a long root, and
$\langle\lambda,\check{\alpha}\rangle=1$, then the restriction of
$\calL_G$ on $S_{\check{\alpha}}s^\lambda$ is of degree one.
\end{lem}
\begin{rmk}\label{nonsimply-laced} In the proof of Theorem \ref{reduced main theorem },
we need to a lot of rational curves of degree one (or
equivalently, projective lines) in $\overline{\Gr}_G^{\omega_i}$.
If $G$ is not simply-laced, then for short root $\check{\alpha}$,
$S_{\check{\alpha}}s^\lambda$ will not be of degree one, and there
will not be enough projective lines for our purpose.
\end{rmk}
\begin{proof} Since $\calL_G$ is very ample, we use it to embed
$\Gr_G$ into projective spaces. Recall from Proposition
\ref{Borel-Weil},
\[H^0(\Gr_G,\calL_G)^*=\bigoplus_{\gamma\in\pi_1(G)}L(\Lambda+\iota\omega_{i_\gamma})\]
This representation of $\hat{\frakg}$ can be integrated to a
projective action of $G_\K$ on it. More precisely, let $\tilde{G}$
be the simply-connected cover of $G$. Since for any
$\lambda\in\Lambda_G$, $Ad_{t^\lambda}$ acts on $\tilde{G}_\K$,
one can define
\[\widetilde{G_\K}=\tilde{G}_\K\rtimes\Lambda_G/\tilde{G}_\K\rtimes
R_G\] It follows from the construction that the neutral component
$\widetilde{G_\K}^0$ of $\widetilde{G_\K}$ is isomorphic to
$\tilde{G}_\K$ and $\widetilde{G_\K}$ is a $\pi_1(G)$-covering of
$G_\K$. Then similarly to the simply-connected case as observed by
Faltings, and explained in \cite{LS} Proposition 4.3, the
$\hat{\frakg}$-module
$\bigoplus_{\gamma\in\pi_1(G)}L(\Lambda+\iota\omega_{i_\gamma})$
is integrated to a module over the $\G_m$-central extension of
$\widetilde{G_\K}$. Let $v_\Lambda$ be the line of highest weight
vectors in $L(\Lambda)$. Then
$\Gr_G=\widetilde{G_\K}/\tilde{G}_\calO$ is embedded in
$\bbP(\bigoplus_{\gamma\in\pi_1(G)}L(\Lambda+\iota\omega_{i_\gamma}))$
by $g\mapsto gv_\Lambda$ for $g\in\widetilde{G_\K}$. We recall the
weight for the line $s^\mu=t^\mu\cdot v_\Lambda$ is
$\Lambda-\iota\mu-\dfrac{(\mu,\mu)}{2}\check{\delta}$ (cf.
\cite{MV} Proposition 3.1).

Now observe that $S_{\check{\psi}}$ is isomorphic to $SL_2$ or
$PSL_2$. We claim

\vspace{2mm}

\noindent\bf Claim. \rm The Lie algebra of $S_{\check{\psi}}$ is
the $\fraks\frakl_2$-triple spanned by
\[\{e_{n\check{\delta}+\check{\alpha}},2nK/(\check{\alpha},\check{\alpha})^*+\alpha,e_{-n\check{\delta}-\check{\alpha}}\}\]

\vspace{2mm}

Assume this claim for the moment. Since
$n<\langle\lambda,\check{\alpha}\rangle$,
$U_{-\check{\psi}}s^\lambda=s^\lambda$. Therefore,
$S_{\check{\psi}}s^\lambda\cong S_{\check{\psi}}/B_-\cong \bbP^1$,
where $B_-$ is the Borel subgroup of $S_{\check{\psi}}$ containing
$U_{-\check{\psi}}$.

Let $V\subset L(\Lambda)$ be the minimal $\fraks\frakl_2$-stable
subspace generated by $t^\lambda v_\Lambda$. Then $V$ is an
irreducible representation of this $\fraks\frakl_2$ of lowest
weight
\[\langle
\dfrac{2nK}{(\check{\alpha},\check{\alpha})^*}+\alpha,\Lambda-\iota\lambda-\dfrac{(\lambda,\lambda)}{2}\check{\delta}\rangle=\dfrac{2(n-\langle\lambda,\check{\alpha}\rangle)}{(\check{\alpha},\check{\alpha})^*}\]
Now $SL_2$-theory shows that $S_{\check{\psi}}s^\lambda$ is
contained in $\bbP(V)$, with two
$(2nK/(\check{\alpha},\check{\alpha})^*+\alpha)$-fixed points
$s^\lambda$ and
$s^{\lambda-(\langle\lambda,\check{\alpha}\rangle-n)\alpha}$.
Furthermore, $S_{\check{\psi}}s^\lambda$ is rational in $\bbP(V)$
of degree
$\dfrac{2(\langle\lambda,\check{\alpha}\rangle-n)}{(\check{\alpha},\check{\alpha})^*}$.
\end{proof}

It remains to prove the claim above. It is a direct consequence of
the following lemma, which is well-known. We include a proof for
completeness.

\begin{lem}
Let $\check{\psi}=n\check{\delta}+\check{\alpha}$ be a real root,
then the corresponding coroot is
$\psi=\dfrac{2n}{(\check{\alpha},\check{\alpha})^*}K+\alpha$.
\end{lem}

\begin{proof}Observe we could assume $n\geq 0$. The coroot corresponding to the simple root
$\check{\alpha}_0=\check{\delta}-\check{\theta}$ is
$\alpha_0=K-\theta$. We prove the lemma by induction on $n$.

For $n=0$, it is clear. Assuming for $n-1$, the lemma holds.
Observe that since $\check{\theta}$ is a long root, then for any
root $\check{\alpha}\neq\check{\theta}$,
$|\langle\theta,\check{\alpha}\rangle|\leq 1$. If
$\langle\theta,\check{\alpha}\rangle$=1, then
$r_0((n-1)\check{\delta}+\check{\alpha})=n\check{\delta}+\check{\alpha}-\check{\theta}$.
where $r_0$ is the simple reflection in the affine Weyl group
corresponding to the vertex $i_0$ in the affine Dynkin diagram. So
the coroot corresponding to
$n\check{\delta}+\check{\alpha}-\check{\theta}$ is
\[r_0(\frac{2(n-1)K}{(\check{\alpha},\check{\alpha})^*}+\alpha)=\frac{2nK}{(\check{\alpha},\check{\alpha})^*}+\alpha-\frac{2\theta}{(\check{\alpha},\check{\alpha})^*}\]
Observe that the coroot corresponding to
$\check{\alpha}-\check{\theta}$ is exactly
$\alpha-\dfrac{2\theta}{(\check{\alpha},\check{\alpha})^*}$.
Therefore, for $|\langle\theta,\check{\alpha}\rangle|=1$, the
coroot corresponding to $n\check{\delta}+\check{\alpha}$ is
$2nK/(\check{\alpha},\check{\alpha})^*+\alpha$. Now the lemma
follows from for any $\check{\alpha}\in\Delta$, there exists $w\in
W$ and $\check{\beta}\in\Delta,
|\langle\theta,\check{\beta}\rangle|=1$, such that
$\check{\alpha}=w(\check{\beta})$.
\end{proof}

\begin{s} Let $L$ be a Levi subgroup of some parabolic subgroup of
$G$, and $M=[L,L]$ be the derived group of $L$. Let $T_M\to
T_G=T_L$ be the maximal tori of $M$, $G$ and $L$. Then there is a
natural embedding of lattices $\Lambda_M\subset\Lambda_G$. In
addition, we have a projection $p:\Lambda_G\to\Lambda_M$ defined
as follows. Let $Z(L)^0$ be the neutral connected component of the
center of $L$. Since $G$ is assumed to be of adjoint-type, $M$ is
also of adjoint type. Therefore $T_M\times Z(L)^0\to T_G$ is an
isomorphism. Then the projection $p$ is induced from $T_G\cong
T_M\times Z(L)^0\to T_M$.

For $\lambda\in\Lambda_G$, let $\Gr_M\to \Gr_G$ be the embedding,
induced from $M_\K\to \Gr_G$ by sending $m\mapsto ms^\lambda$. We
have the following
\end{s}
\begin{lem}\label{Levi} Under the above embedding, $\overline{\Gr}_M^{p(\lambda)}\subset
\Gr_M\cap\overline{\Gr}_G^\lambda$. If the Dynkin subdiagram for
$M$ contains some vertex $i$ of the Dynkin diagram of $G$
corresponding a long root $\check{\alpha}_i$ of $G$, then the
restriction of $\calL_G$ to $\Gr_M$ is isomorphic to $\calL_M$.
\end{lem}
\begin{proof} Let us prove the second statement. Without loss of generality, we could assume $M$ is the derived group of
a standard Levi factor in $G$. Let
$\mu=\lambda-p(\lambda)+\omega_i$. Then $s^\mu\in \Gr_M$, and
$S_{\check{\alpha}_i}s^\mu\subset \Gr_M$. Furthermore, the
restriction of $\calL_G$ to the rational curve
$S_{\check{\alpha}_i}s^\mu$ is $\calO(1)$ by Lemma \ref{ration
curves}, since $\check{\alpha}_i$ is a long root. Therefore, the
restriction of $\calL_G$ to $\Gr_M$ must be $\calL_M$.
\end{proof}

\begin{s}\label{addimisible} Recall the definition of
$\sigma_\lambda$ from Proposition \ref{surjectivity}. Define
$U^{\mu}_\lambda=\overline{\Gr}_G^{\lambda}-\mbox{supp}(\sigma_{\mu})$.
It is clear that $U^{\mu}_\lambda$ is an affine open subscheme of
$\overline{\Gr}_G^{\lambda}$ containing $s^\mu$ as the unique
$T$-fixed point. Therefore, $\{U^{\mu}_\lambda\}$ for all $\mu$
such that $s^\mu\in \overline{\Gr}_G^\lambda$ form an open cover
of $\overline{\Gr}_G^{\lambda}$, for the complement of the union
of these open subsets is closed without any $T$-fixed point and
thus must be an empty set. We will denote $U^\lambda_\lambda$
simply by $U^\lambda$. Let $N^\lambda$ be the unipotent subgroup
of $G_\calO$ generated by $U_{n\check{\delta}+\check{\alpha}}$
with $\check{\alpha}$ positive root and $0\leq
n<\langle\lambda,\check{\alpha}\rangle$. Then $N^\lambda\cdot
s^\lambda=U^\lambda$.

We will call an $m$-dimensional $T$-invariant subvariety $Z$ in
$\overline{\Gr}_G^\lambda$ admissible, if there exists some affine
open subset of $Z$, which contains some open subset of
$U_{n_1\check{\delta}+\check{\alpha}_{i_1}}\cdots
U_{n_m\check{\delta}+\check{\alpha}_{i_m}}s^\mu$, where
$s^\mu\in\overline{\Gr}_G^\lambda$,
$n_1\check{\delta}+\check{\alpha}_{i_1},\ldots,n_m\check{\delta}+\check{\alpha}_{i_m}$
are distinct, and $n_j\geq 0$. Observe that any $T$-invariant
curve in $\Gr_G^\lambda$ is admissible. Furthermore, $U^\lambda$
is admissible.
\end{s}

From now on, we assume that $G$ is simple simply-laced algebraic
group. We will prove Theorem \ref{reduced main theorem } by case
by case considerations. We embed $\Gr_G$ into the projective space
by $\calL_G$. The way we label the Dynkin diagram follows
Bourbaki's notation (cf. \cite{Bour}, PLATE).

\begin{prop}\label{Minuscule}
Theorem \ref{reduced main theorem } holds if $\lambda$ is a
minuscule coweight.
\end{prop}
\begin{proof}
We shall point out that the proof presented here is not the
simplest one. The reason, however, that we adapt it here is two
folded. First, the proof itself is self-contained, and requires no
knowledge about minuscule representations. Secondly and more
importantly, it serves as a prototype for the discussions for
non-minuscule fundamental coweights.

Recall that if $\lambda$ is minuscule, then
$\langle\lambda,\check{\alpha}\rangle\leq 1$ for any
$\check{\alpha}$ positive root. Denote
$\Delta_\lambda=\{\check{\alpha}\in\Delta_+,\langle\lambda,\check{\alpha}\rangle=1\}$.
In this case $\overline{\Gr}_G^\lambda=\Gr_G^\lambda=G/P_\lambda$,
where $P_\lambda$ is the parabolic subgroup corresponding to
$\lambda$, i.e. $P_\lambda$ is generated by $B^-$ and
$U_{\check{\alpha}}$ for $\check{\alpha}$ positive root not
contained in $\Delta_\lambda$. Let $\sigma\in
H^0(\Gr_G^\lambda,\I^\lambda(1))$, then $\sigma$ vanishes on all
$T$-fixed points $s^\mu$, with $\mu=w\lambda$ for some $w\in W$.
For those $\check{\alpha}$ such that
$\langle\lambda,\check{\alpha}\rangle=1$, by Lemma \ref{ration
curves}, $S_{\check{\alpha}}\cdot s^\lambda$ is a rational curve,
with two $T$-fixed points $s^\lambda$ and $s^{\lambda-\alpha}$.
Since the restriction of $\calL_G$ on $S_{\check{\alpha}}\cdot
s^\lambda$ is isomorphic to $\calO(1)$ and $\sigma$ vanishes at
$s^\lambda$ and $s^{\lambda-\alpha}$,
$\sigma|_{S_{\check{\alpha}}\cdot s^\lambda}=0$. Same argument
shows that $\sigma$ in fact vanishes along any $T$-invariant
curves in $\Gr_G^\lambda$.

Assume that $\sigma$ vanishes along those $(r-1)$-dimensional
admissible $T$-invariant subvarieties in $\Gr_G^\lambda$. We prove
it also vanishes along those $r$-dimensional admissible
$T$-invariant subvarieties.  Let $p$ be a closed point of
$\Gr_G^\lambda$ contained in some $r$-dimensional admissible
$T$-invariant subvariety. Without loss of generality, we could
assume $p\in U_{\check{\alpha}_{i_1}}\cdots
U_{\check{\alpha}_{i_r}} s^\lambda$, with
$\check{\alpha}_{i_1},\ldots,\check{\alpha}_{i_r}$ distinct, and
$\check{\alpha}_{i_r}\in\Delta_\lambda$. Therefore, one could
write $p=gp'$ with $g\in U_{\check{\alpha}_{i_1}}\cdots
U_{\check{\alpha}_{i_{r-1}}}$ and
$p'=\exp(ce_{\check{\alpha}_{i_r}}) s^\lambda$ for some $c\in \C$.
Let $C=gS_{\check{\alpha}_{i_r}}s^\lambda$ be the translation of
$S_{\check{\alpha}_{i_r}}s^\lambda$ by $g$. Then $C$ is a rational
curve of degree one in $\Gr_G^\lambda$, for $\calL_G$ is
$G_\calO$-linearized. Observe that $p\in
gS_{\check{\alpha}_{i_r}}s^\lambda$, and that $gs^\lambda$ and
$gs^{\lambda-\alpha_{i_r}}$ are also contained in this curve.
Furthermore, they belong to $(r-1)$-dimensional admissible
$T$-invariant subvarieties of $\Gr_G^\lambda$. Since $\sigma$
vanishes at $gs^\lambda$ and $gs^{\lambda-\alpha_{i_r}}$ by
induction, $\sigma$ vanishes along $C$, and in particular, at $p$.
\end{proof}

\begin{thm} Theorem \ref{Main Theorem} holds for simple algebraic groups of type
$A$.
\end{thm}
\begin{proof} Every fundamental coweight is minuscule for
simple algebraic groups of type $A$.
\end{proof}

\begin{prop}\label{higest coroot}
Assume $G$ is a simple algebraic group of type $A$, $D$, or $E$.
Then Theorem \ref{reduced main theorem } holds if $\lambda$ is the
highest coroot.
\end{prop}
\begin{proof}
Recall that in the case $G$ is simply-laced, the highest coroot is
just $\theta$. Let $P$ be a minimal parabolic subgroup of $G$
contains $U_{\check{\theta}}$ and $U_{-\check{\theta}}$. There is
an obvious Levi subgroup of $P$ whose derived group is $M\cong
PSL_2$, and
$\mbox{Lie}M=\C\{e_{\check{\theta}},\theta,e_{-\check{\theta}}\}$.
Let $\Gr_M\to \Gr_G$ be the embedding, induced from $M_\K\to
\Gr_G$ by sending $m\mapsto ms^\theta$. Let
$Z_\theta=\overline{M_\calO s^\theta}\subset \Gr_M\cap
\overline{\Gr}_G^\theta$. Then
$Z_\theta\cong\overline{\Gr}_M^{p(\theta)}\cong\overline{\Gr}_{SL_2}^2$,
where we identify the coweight lattice of $\fraks\frakl_2$ with
$\Z$ by identifying the fundamental coweight with 1. This is
$T$-invariant. Corollary 1.2.8 implies that
$Z_\theta^T\cong(\overline{\Gr}_{SL_2}^2)^{T_{SL_2}}$.
Furthermore, the restriction of $\calL_G$ on
$\overline{\Gr}_G^\theta$ to $\overline{\Gr}_{SL_2}^2$ is
$\calL_{SL_2}$ by Lemma \ref{Levi}, since
$(\check{\theta},\check{\theta})^*=2$. Now let $\sigma\in
H^0(\overline{\Gr}_G^\theta,\I^\theta(1))$. Since Theorem
\ref{reduced main theorem } holds for $SL_2$. One obtains that
$\sigma|_{Z_\theta}=0$. Likewise, one can define $Z_\alpha$ for
any $\alpha$ positive coroot, and for the same reason,
$\sigma|_{Z_\alpha}=0$.

Now observe that $\langle\theta,\check{\alpha}\rangle\leq 1$ for
any $\check{\alpha}$ positive root other than $\check{\theta}$,
while $\langle\theta,\check{\theta}\rangle=2$. Denote
$\Delta_\theta=\{\check{\alpha} \mbox{ positive, }
\langle\check{\alpha},\theta\rangle=1\}$. Then one can use the
same method as in \ref{Minuscule} to prove that $\sigma$ indeed
vanishes along any $r$-dimensional admissible $T$-invariant
subvarieties in $\overline{\Gr}_G^\theta$. If $s^\alpha$ and
$s^\beta$ are connected by some $T$-invariant curve, then this
curve is rational of degree 1 unless $\alpha=-\beta$. For the
former case, $\sigma=0$ at $s^\alpha$ and $s^\beta$ implies that
$\sigma$ vanishes along this curve. For the latter case, the
rational curve is of degree 2, and is contained in $Z_\alpha$.
This proves the case $r=1$. $r>1$ is proved by induction as in
\ref{Minuscule}. Let $p$ be a closed point of
$\overline{\Gr}_G^\theta$ contained in some $r$-dimensional
addmissible $T$-invariant subvariety. Since
$\overline{\Gr}_G^\theta-\Gr_G^\theta=s^0$, we could assume $p\in
\Gr_G^\theta$. Observe that, as in \ref{Minuscule}, we could
further assume that \[p\in
U_{\check{\psi}_1}U_{\check{\psi}_2}U_{\check{\psi}_3}\cdots
U_{\check{\psi}_r} s^\theta\] where $\check{\psi}_i$ are distinct
and,
$\check{\psi}_r\in\{\check{\theta},\check{\delta}+\check{\theta}\}\cup\Delta_\theta$.
If $\check{\psi}_r\neq\check{\theta}$, one writes $p=gp'$ with
$g\in U_{\check{\psi}_1}\cdots U_{\check{\psi}_{r-1}}$ and
$p'=\exp(ce_{\check{\psi}_r})s^\theta$ for some $c\in \C$. Since
$\check{\psi}_r\neq\check{\theta}$, $p$ is contained in the degree
one rational curve $gS_{\check{\psi}_r}s^{\theta}$. (Observe that
$S_{\check{\delta}+\check{\theta}}s^\theta$ is rational of degree
one by Lemma \ref{ration curves}.) Now argue as in \ref{Minuscule}
to show $\sigma$ vanishes at $p$. Therefore, one could assume that
$\check{\psi}_r=\check{\theta}$. Let
$\check{\psi}_{r-1}=n_{r-1}\check{\delta}+\check{\beta}_{r-1}$. We
could further assume that
$\langle\check{\beta}_{r-1},\theta\rangle<0$. Otherwise, we can
simply interchange $U_{\check{\psi}_{r-1}}$ and
$U_{\check{\psi}_r}$ (since
$[e_{\check{\theta}},e_{n_{r-1}\check{\delta}+\check{\beta}_{r-1}}]=0$
in this case), and return to the case that
$\check{\psi}_r\neq\check{\theta}$. However, if
$\langle\check{\beta}_{r-1},\theta\rangle<0$,
\[U_{n_{r-1}\check{\delta}+\check{\beta}_{r-1}}U_{\check{\theta}}s^\theta\subset
U_{\check{\theta}}U_{n_{r-1}\check{\delta}+\check{\beta}_{r-1}+\check{\theta}}s^\theta\]
We still return to the case that
$\check{\psi}_r\neq\check{\theta}$.
\end{proof}

\begin{s}\label{subschubert} We need some more preparations in order to prove Theorem \ref{reduced main theorem } for algebraic groups of
type $D$. We temporarily do not assume that $G$ is simply-laced.
For any $\lambda\in\Lambda_G^+$, there is an embedding
$G/P_{\lambda}\cong G\cdot
s^\lambda\subset\overline{\Gr}_G^\lambda$, where $P_\lambda$ is
the parabolic subgroup of $G$ generated by $B^-$ and
$U_{\check{\alpha}}$ with $\check{\alpha}$ positive and
$\langle\check{\alpha},\lambda\rangle=0$. $s^{w\lambda}\in
G/P_\lambda$ will simply be denoted by $w\in G/P_\lambda$.

It is not difficult to see that the restriction of $\calL_G$ to
$G/P_{\lambda}$ is
$\calO(\iota\lambda)=G\times^{P_\lambda}\C_{\iota\lambda}$, where
$P_\lambda$ acts on $\C_{\iota\lambda}$ through the character
$\iota\lambda$. Therefore $\Gamma(G/P_\lambda,\calL_G)\cong
V^{\iota\lambda}$ by the Borel-Weil theorem. We have the natural
map
$\Gamma(\overline{\Gr}_G^\lambda,\calL_G)\to\Gamma(G/P_{\lambda},\calO(\iota\lambda))$.
Assume that $\lambda$ is not minuscule. Let
$Z=\overline{\Gr}_G^\lambda\setminus \Gr_G^\lambda$ endowed with
the reduced scheme structure and $\J$ be the sheaf of ideal
defining $Z$. It is clear that the composition
\[\Gamma(\overline{\Gr}_G^\lambda,\J\otimes\calL_G)\hookrightarrow\Gamma(\overline{\Gr}_G^\lambda,\calL_G)\twoheadrightarrow\Gamma(G/P_{\lambda},\calO(\iota\lambda))\]
is non-zero (e.g. $\sigma_\lambda\neq 0$ at $s^\lambda$) and
therefore surjective. Now assume that $G$ is simply-laced. We have
\end{s}
\begin{prop}Let $\lambda\in\Lambda_G^+$, which is not minuscule. If $\langle\lambda,\check{\alpha}\rangle\leq 2$ for
any positive root $\check{\alpha}$, then
$\Gamma(\overline{\Gr}_G^\lambda,\J\otimes\calL_G)\to\Gamma(G/P_{\lambda},\calO(\iota\lambda))$
is an isomorphism.
\end{prop}

\begin{proof} Assume that
$\sigma\in\Gamma(\overline{\Gr}_G^\lambda,\calL_G)$ with
$\sigma|_{Z\cup G\cdot s^\lambda}=0$. We prove that $\sigma=0$.
Recall that $\Gr_G^\lambda$ is an affine bundle over $G/P_\lambda$
whose fiber at $gs^\lambda$ ($g\in G$) is
$g\prod_{\langle\check{\alpha},\lambda\rangle=2}U_{\check{\delta}+\check{\alpha}}s^\lambda$.
Observe that each $gS_{\check{\delta}+\check{\alpha}}s^\lambda$ is
a rational curve of degree one containing $gs^\lambda\in
G/P_\lambda$ and $gs^{\lambda-\alpha}\in Z$, at which $\sigma$
vanishes. Our method used in the proofs of Proposition
\ref{Minuscule} and Proposition \ref{higest coroot} indicates that
$\sigma$ vanishes along the fiber over any $p\in G/P_\lambda$.
\end{proof}

\begin{thm}\label{type D} Theorem \ref{reduced main theorem } holds for simple algebraic groups of type $D$.
\end{thm}
\begin{proof} Observe in
this case, $\omega_1,\omega_{\ell-1},\omega_\ell$ are minuscule,
$\omega_2$ is the highest coroot. Therefore, the theorem holds for
$D_4$. Now assume that $\ell>4$. We are aiming to prove that
Theorem \ref{reduced main theorem } holds for $\omega_i, 3\leq
i\leq\ell-2$. Let
\[\check{\beta}_1=\check{\alpha}_1,\check{\beta}_2=\check{\alpha}_2,\ldots,\check{\beta}_{i-1}=\check{\alpha}_{i-1},\check{\beta}_i=\check{\alpha}_{i-1}+2\check{\alpha}_i+\cdots+2\check{\alpha}_{\ell-2}+\check{\alpha}_{\ell-1}+\check{\alpha}_\ell\]
Then $\check{\beta}_1,\ldots,\check{\beta}_i$ determines a
subgroup of $G$, of type $D_i$, which we denote by $M$. One can
show that $M$ is the derived group of a Levi subgroup of $G$.
Indeed, another set of simple roots
$\{\check{\alpha}'_1,\ldots,\check{\alpha}'_\ell\}$ of $G$ can be
chosen as
\[-\check{\alpha}_{i+1},\ldots,-\check{\alpha}_{\ell-1},-\check{\alpha}_1-\cdots-\check{\alpha}_{\ell-2}-\check{\alpha}_\ell,\check{\beta}_1,\ldots,\check{\beta}_i\]
Then $M$ is the derived group of a standard Levi factor for this
set of simple roots. Let $T_M$ be the maximal torus of $M$, whose
Lie algebra is generated by $\beta_1,\ldots,\beta_i$. Denote
$\Delta_i=\{\check{\alpha},\langle\omega_i,\check{\alpha}\rangle=2\}$.
Observe that
\[\prod_{\check{\alpha}\in\Delta_i}U_{\check{\alpha}}U_{\check{\delta}+\check{\alpha}}s^{\omega_i}=M_\calO
s^{\omega_i}\] We denote its closure by $Z_{\omega_i}$. Then
$Z_{\omega_i}\cong\overline{\Gr}_M^{2\omega^M_i}$, where
$\omega^M_i$ is the fundamental coweight of $M$ corresponding to
$\check{\beta}_i$. We have: (i) the pullback of $\calL_G$ to
$Z_{\omega_i}$ is isomorphic to $\calL_M$ by Lemma \ref{Levi};
(ii) $(Z_{\omega_i})^T\cong
(\overline{\Gr}_M^{2\omega^M_i})^{T_M}$ by Corollary
\ref{fixed-point in Levi}.

Let $\sigma\in H^0(\overline{\Gr}_G^{\omega_i},\I^{\omega_i}(1))$,
then by induction for $\ell$, $\sigma|_{Z_{\omega_i}}=0$.
Likewise, for any $w\omega_i$, one can define $Z_{w\omega_i}$ and
prove similarly that $\sigma|_{Z_{w\omega_i}}=0$. The goal is to
show that $\sigma=0$ and therefore
$H^0(\overline{\Gr}_G^{\omega_i},\I^{\omega_i}(1))=0$. Since
$H^0(\overline{\Gr}_G^{\omega_i},\I^{\omega_i}(1))$ is a
$T$-module, we could assume that $\sigma$ is a $T$-weight vector.

Recall the settings as in \ref{subschubert}. In this case we have
$\J$ the ideal defining $\overline{\Gr}_G^{\omega_{i-2}}$ in
$\overline{\Gr}_G^{\omega_i}$ and
$\Gamma(\overline{\Gr}_G^{\omega_i},\J\otimes\calL_G)\cong\Gamma(G/P_i,\calO(\check{\omega}_i))$,
where $P_i=P_{\omega_i}$ is the parabolic subgroup generated by
$B_-$ and $U_{\check{\alpha}}$ with
$\langle\omega_i,\check{\alpha}\rangle=0$. By induction on $i$, we
could assume that
$\sigma\in\Gamma(\overline{\Gr}_G^{\omega_i},\J\otimes\calL_G)$.
Then our theorem is a consequence of the following proposition.
\end{proof}

\begin{prop} Let $\sigma\in\Gamma(G/P_i,\calO(\check{\omega}_i))$.
If $\sigma$ is a $T$-weight vector (or equivalently,
$\mathrm{supp}(\sigma)$ is $T$-invariant), and vanishes along any
\[Z_w:=(\prod_{\langle
w\omega_i,\check{\alpha}\rangle=2}U_{\check{\alpha}})w, \ \ \ w\in
W\] then $\sigma=0$.
\end{prop}
\begin{proof} Recall that
$\Gamma(G/P_{\omega_i},\calL_G)=V^{\check{\omega}_i}$. The
anti-dominant weights appearing in $V^{\check{\omega}_i}$ are
$\check{\omega}_i,\check{\omega}_{i-2},\ldots$. Thus, without loss
of generality, we could assume that $\sigma\in
V^{\check{\omega}_i}(\check{\omega}_{i-2k})$ where $2k\leq i$. We
will deduce the proposition from the following lemma, whose proof
is left to the readers.  For any $w\in W$, the
$(-w\check{\omega}_i)$-weight space
$V^{\check{\omega}_i}(-w\check{\omega}_i)$ is one dimensional.
Pick up an extreme vector $0\neq v_w\in
V^{\check{\omega}_i}(-w\check{\omega}_i)$ for each $w$.
\begin{lem} $V^{\check{\omega}_i}(-\check{\omega}_{i-2k})$ has a
basis of the forms
\[e_{\check{\beta}_{d,1}}e_{\check{\beta}_{d,2}}\cdots
e_{\check{\beta}_{d,k}}v_{w_d},\ \ \ w_d\in W, d=1,2,\ldots,\dim
V^{\check{\omega}_i}(-\check{\omega}_{i-2k})\] with
$(\check{\beta}_{d,j},w_d\check{\omega}_i)^*=2$ and
$(\check{\beta}_{d,j},\check{\beta}_{d,j'})^*=0$ for $j\neq j'$.
\end{lem}
Observe that $(V^{\check{\omega}_i})^*\cong V^{\check{\omega}_i}$
via the inner product. Choose a basis of
$V^{\check{\omega}_i}(-\check{\omega}_{i-2k})$ of the forms
$u_d=e_{\check{\beta}_{d,1}}e_{\check{\beta}_{d,2}}\cdots
e_{\check{\beta}_{d,k}}v_{w_d}$ as in the lemma. Let $\{u^*_d\}$
be the basis of
$V^{\check{\omega}_i}(-\check{\omega}_{i-2k})^*=V^{\check{\omega}_i}(\check{\omega}_{i-2k})$
dual to $\{u_d\}$ and write $\sigma=\sum_d\sigma_du^*_d$ where
$\sigma_d\in\C$. Regard
$u_d^*\in\Gamma(G/P_i,\calO(\check{\omega}_i))$ by embedding
$G/P_i$ into $\bbP(V^{\check{\omega}_i})$ through $g\mapsto gv_1$,
then $u^*_d$ does not vanish along the whole
$U_{\check{\beta}_{d,1}}U_{\check{\beta}_{d,2}}\cdots
U_{\check{\beta}_{d,k}}w_d$, while $u^*_{d'}$ does vanish along it
for any $d'\neq d$. Since $\sigma$ vanishes on
$U_{\check{\beta}_{d,1}}U_{\check{\beta}_{d,2}}\cdots
U_{\check{\beta}_{d,k}}w_d$ for all $d$, $\sigma=0$.
\end{proof}

\begin{s}\label{type E,I}
The main Theorem \ref{Main Theorem} now is proved for $G$ of type
$A$ or $D$. Let us also discuss some cases that we can prove now
for simple algebraic groups of type $E$.

For $G$ being of type $E_6$, $\omega_1$ and $\omega_6$ are
minuscule, and $\omega_2$ is the highest coroot; for $G$ being of
type $E_7$, $\omega_1$ is the highest coroot, $\omega_7$ is
minuscule; for $G$ being of type $E_8$, $\omega_8$ is the highest
coroot. We could also prove
\end{s}

\begin{prop}\label{type E,II} Theorem \ref{reduced main theorem } also holds in the following cases:
(i) $G$ is of type $E_6$, and $\lambda=\omega_3$ or $\omega_5$;
(ii) $G$ is of type $E_7$, and $\lambda=\omega_2$ or $\omega_6$;
(iii) $G$ is of type $E_8$, and $\lambda=\omega_1$.
\end{prop}
\begin{proof} The proofs are similar to the previous ones. Take
$G=E_6$ and
 $\lambda=\omega_3$ for example. Let
\[\check{\beta}_1=\check{\alpha}_1+2\check{\alpha}_2+2\check{\alpha}_3+\check{\alpha}_4+\check{\alpha}_6,\check{\beta}_2=\check{\alpha}_5,\check{\beta}_3=\check{\alpha}_4,\check{\beta}_4=\check{\alpha}_3,\check{\beta}_5=\check{\alpha}_6\]
Then $\{\check{\beta}_1,\ldots,\check{\beta}_5\}$ determines a
subgroup $M$ of $G$, of type $A_5$, which is the derived group of
a Levi subgroup of $G$. (To see this, observe that
$\{\check{\beta}_1,\check{\beta}_2,\check{\beta}_3,\check{\beta}_4,\check{\beta}_5,-\check{\alpha}_2-2\check{\alpha}_3-2\check{\alpha}_4-\check{\alpha}_5-\check{\alpha}_6\}$
can be chosen as the set of simple roots for $G$.) Denote
$\Delta_3=\{\check{\alpha},\langle\omega_3,\check{\alpha}\rangle=2\}$.
Then the closure of
$\prod_{\check{\alpha}\in\Delta_3}U_{\check{\alpha}}U_{\check{\delta}+\check{\alpha}}s^{\omega_3}$
in $\overline{\Gr}_G^{\omega_2}$, is isomorphic to
$\overline{\Gr}_M^{2\omega^M_1}$, where $\omega^M_1$ is the
fundamental coweight of $M$ corresponding to $\check{\beta}_1$.
Now proceed as in \ref{type D}.
\end{proof}

\subsection{Applications: the smooth locus of
$\overline{\Gr}_G^\lambda$}\label{simple app} We prove Corollary
\ref{Smooth locus} in this section. It is clear that to prove the
corollary, it is enough to prove that $\overline{\Gr}_G^\lambda$
is not smooth at $s^\mu$ for $\mu\in\Lambda_G^+,\mu<\lambda$.
Since $(\overline{\Gr}_G^\lambda)^T$ is a finite scheme, a direct
consequence of Theorem \ref{Main Theorem} is the following

\begin{cor}\label{length}
In the notations of Theorem \ref{Main Theorem}, let
$V_\lambda=H^0(\overline{\Gr}_G^\lambda,\calL_G)^*$ be the affine
Demazure module. Then
\[\dim V_\lambda(-\iota\mu)=\mathrm{length}_\C\calO_{(\overline{\Gr}_G^\lambda)^T,s^\mu}\]
where $V_\lambda(-\iota\mu)$ denotes the $(-\iota\mu)$-weight
subspace of $V_\lambda$.
\end{cor}
\begin{proof} Observe that under the isomorphism in Theorem \ref{Main Theorem},
$\calO_{(\overline{\Gr}_G^\lambda)^T,s^\mu}$ corresponds to the
$(\iota\mu)$-weight subspace in
$\Gamma(\overline{\Gr}_G^\lambda,\calL_G)$.
\end{proof}

To prove that $\overline{\Gr}_G^\lambda$ is not smooth at $s^\mu$,
we will need the following simple and well-known lemmas.
\begin{lem} Let $X=\spec A$ be an affine scheme over $k$ with an action
of a torus $T$. Then $A$ is an algebraic representation of $T$.
Decompose $A=A_0\oplus A_+$, where $A_0$ is the zero weight space
and $A_+$ is the direct sum of nonzero weight spaces. Let $X^T$ be
the $T$-fixed subcheme of $X$ defined as above and $I$ the ideal
defining $X^T$. Then $I=A_+A$.
\end{lem}

\begin{lem}\label{smoothness} Let $X$ be a smooth quasi-projective variety with an action
of a torus $T$. Then $X^T$ is also smooth, and in particular, is
reduced.
\end{lem}
\begin{proof} Since $X$ is covered by $T$-invariant open affine subschemes, we could
assume that $X$ is affine. Let $x\in X^T$, then $\spec\calO_{X,x}$
is a $T$-invariant subscheme of $X$. Therefore, we could replace
our $X$ by $\spec A$, where $A=\calO_{X,x}$. Let $\frakm$ be the
maximal ideal of $A$. Then one can find a $T$-invariant subspace
$N$ in $\frakm$, such that the natural projection
$N\to\frakm/\frakm^2$ is a $T$-module isomorphism. Decompose
$N=N_0\oplus N_+$, where $N_0$ is the zero weight space and $N_+$
is the direct sume of some nonzero weight spaces. Since $A$ is
regular, $N$ generate $\frakm$. By the lemma above, it is easy to
see that $X^T=\spec (A/N_+)$, which is obviously regular.
\end{proof}

Now, according to Corollary \ref{length} and Lemma
\ref{smoothness}, to prove Corollary \ref{Smooth locus}, it is
enough to show that

\begin{lem} For a simple algebraic group $G$ of type $A$ or $D$, let $\lambda,\mu\in\Lambda_G^+,\mu\leq\lambda$, then $\dim
V_\lambda(-\iota\mu)>1$.
\end{lem}
\begin{proof} We need the following

\begin{sublem}For any simple, not necessarily simply-laced algebraic group $G$, the natural morphism
$H^0(\overline{\Gr}_G^\lambda,\calL_G^{\otimes k})\to
H^0(\overline{\Gr}_G^\mu,\calL_G^{\otimes k})$ is surjective.
\end{sublem}

Assuming the above sublemma, we first conclude the proof of the
proposition, and in consequence, Corollary \ref{Smooth locus}.
Recall that $\tilde{G}$ is denoted the simply-connected cover of
$G$. The sublemma implies that as $\tilde{G}$-modules, $V_\mu$ is
a direct summand of $V_\lambda$. However, $V_\lambda$ is a
$\tilde{G}$-module of lowest weight $-\iota\lambda$. Therefore,
$V_\lambda$ contains the simple $\tilde{G}$-module
$V^{-w_0\iota\lambda}$ of highest weight $-w_0\iota\lambda$ as a
direct summand., where $w_0$ is the element in the Weyl group of
maximal length. Since $V_\mu$ is of lowest weight $-\iota\mu$,
$V_\mu\cap V^{-w_0\iota\lambda}=\emptyset$ in $V_\lambda$. Observe
that $\iota\mu\in\check{\Lambda}_+$ and
$\iota\mu\leq\iota\lambda$. Whence, $\dim
V^{-w_0\iota\lambda}(-\iota\mu)\geq 1$. This together with $\dim
V_\mu(-\iota\mu)=1$ implies the lemma.

It remains to prove the sublemma. When $G$ is of type $A$ or $D$
and $k=1$, which is the case we need, the sublemma is easily
proved as follows. Apply the main theorem. Then the surjectivity
of $H^0(\overline{\Gr}_G^\lambda,\calL_G)\to
H^0(\overline{\Gr}_G^\mu,\calL_G)$ is equivalent to the
surjectivity of
$\calO_{(\Gr_G^\lambda)^T}\to\calO_{(\Gr_G^\mu)^T}$, which is
obvious.

However, let us include a proof of the full statement of the
sublemma for completeness. This is well-known and in fact is the
one of the ingredients for the proof of the Demazure character
formula in \cite{Ku} and \cite{Ma}. (Therefore, the previous easy
proof for $G$ simply-laced and $k=1$ does not quite apply, since
we used the main theorem, whose proof relies on \cite{Ku} and
\cite{Ma}.) The statement of the sublemma is the direct
consequence of the following two facts. First, there exists a flat
model $\overline{\Gr}_G^\lambda$ over $\Z$ such that: (i) if
$\mu\leq\lambda$, $\overline{\Gr}_G^\mu$ is a closed subscheme of
$\overline{\Gr}_G^\lambda$; (ii) over an open subset of $\spec\Z$,
the geometrical fibers are reduced and therefore are the Schubert
varieties over that base field (e.g. see \cite{MR} \S 3, Lemma 3).
Second, over an algebraically closed field of characteristic $p$,
$\overline{\Gr}_G^\mu\subset\overline{\Gr}_G^\lambda$ are
compatibly Frobenius splitting. This is in fact proved in \cite{F}
and \cite{Ma} for Schubert varieties in the affine flag variety
$\F\ell_G:=G_\K/I$, where $I$ is the Iwahori subgroup. However, it
is known that for any $\overline{\Gr}_G^\lambda$, there exist
Schubert varieties $X_{w_\mu}\subset X_{w_\lambda}$ in $\F\ell_G$
such that the projection
$X_{w_\lambda}\to\overline{\Gr}_G^\lambda$ is proper birational,
and under the projection, the scheme theoretical image of
$X_{w_\mu}$ is $\overline{\Gr}_G^\mu$. The normality of
$\overline{\Gr}_G^\lambda$ (cf. \cite{F} Theorem 8) together with
\cite{MR} \S 1, Proposition 4 implies that
$\overline{\Gr}_G^\mu\subset\overline{\Gr}_G^\lambda$ are
compatibly Frobenius splitting. Then by \cite{MR} \S 1,
Proposition 3, $H^0(\overline{\Gr}_G^\lambda,\calL_G^{\otimes
k})\to H^0(\overline{\Gr}_G^\mu,\calL_G^{\otimes k})$ is
surjective over $\bar{\mathbb F}_p$. Now the first statement and
the semi-continuity theorem (\cite{EGAIII} Theorem 7.7.5) imply
that the sublemma also holds in characteristic 0.
\end{proof}
\section{The bosonic realization}\label{Boson}
We will discuss the geometrical form of the FKS in this section.
The language we will be using is the factorization algebras
developed by Beilinson and Drinfeld in \cite{BD1} \S 3.4 (see also
\cite{FB} Chapter 20). We refer the readers to \cite{FB} Chapter
19 and 20 for the dictionaries between vertex algebras and
factorization algebras. We will fix a smooth curve $X$ throughout
this section. The canonical sheaf $\omega_{X^n}$ on $X^n$ is
regarded as a right $D$-module. For any morphism $f:M\to N$, $f_!$
is denoted the push-forward of right $D$-modules.

\subsection{The Heisenberg algebras}\label{Hei}
\begin{s}\label{Heisenberg} Let $\frakt$ be an abelian Lie algebra, and $(\cdot,\cdot)$ be a symmetric bilinear form on $\frakt$.
Given such data, we define the Heisenberg Lie algebra to be the
central extension
\[0\to\C K\to\hat{\frakt}\to\frakt\hat{\otimes}\K\to 0\]
with the Lie bracket given by
\[[A\otimes f,B\otimes g]=(A,B)\res (gdf)K \ \ \ \mbox{ for } A,B\in\frakt, f,g\in\K\]

For any $\lambda\in\frakt$, the level $k$ Fock module
$\pi^k_\lambda$ of $\hat{\frakt}$ is defined as
$\pi_\lambda=\ind_{\frakt\hat{\otimes}\calO\oplus\C
K}^{\hat{\frakt}}\C$, where $K$ acts on $\C$ by multiplication by
$k$, $\frakt\hat{\otimes}\frakm$ acts on $\C$ by zero ($\frakm$
the maximal ideal of $\calO$), and
$\frakt=\frakt\hat{\otimes}\calO/\frakt\hat{\otimes}\frakm$ acts
on $\C$ by $\iota\lambda$, $\iota:\frakt\to\frakt^*$ being the
isomorphism induced from the non-degenerate bilinear form on
$\frakt$. Level one Fock modules $\pi^1_\lambda$ are usually
simply denoted by $\pi_\lambda$. It is clear that the Fock modules
are irreducible $\hat{\frakt}$-modules if the bilinear form
$(\cdot,\cdot)$ is non-degenerate.

Given a simple algebraic group $G$, we obtain the data
$(\frakt,(\cdot,\cdot))$, where $\frakt$ is a Cartan subalgebra of
$\frakg$, and $(\cdot,\cdot)$ is the restriction to $\frakt$ the
normalized invariant form on $\frakg$. Then the Heisenberg algebra
$\hat{\frakt}$ associated to this data is just restriction of the
central extension of $\frakg\hat{\otimes}\K$ to
$\frakt\hat{\otimes}\K$. Therefore, any module over $\hat{\frakg}$
is a module over $\hat{\frakt}$.

All the discussions in this section are based on the following
theorem.
\end{s}

\begin{thm} Let $G$ be a simple (not necessarily simply-connected) algebraic group of type $A$, $D$,
or $E$. For any $\gamma\in\pi_1(G)$, recall $\omega_{i_\gamma}$
from \ref{minuscule coweight}. Then one has the isomorphism
\[L(\Lambda+\iota\omega_{i_\gamma})\cong\bigoplus_{\lambda\in R_G}\pi_{\omega_{i_\gamma}+\lambda}\]
as $\hat{\frakt}$-modules.
\end{thm}
\begin{proof} By Proposition \ref{Borel-Weil} and Theorem \ref{Fixed-point}, the proposition is equivalent to
prove that the natural morphism
\[\Gamma((\Gr_G)^\gamma,\calL_G)\longrightarrow\Gamma((\Gr_G)^\gamma,\calO_{((\Gr_G)^\gamma)^T}\otimes\calL_G)\]
is an isomorphism. If $G$ is of type $A$ or $D$, this directly
follows from Theorem \ref{Main Theorem} since
$(\Gr_G)^\gamma=\lim\limits_{\longrightarrow}\overline{\Gr}_G^\lambda$,
where limit is taken over $(\omega_{i_\gamma}+\Lambda_G^+)$. To
give a proof of the proposition also applicable to type $E$, we
make use of the following simple observation

\vspace{2mm}

\noindent\bf Claim. \rm For any $\lambda$ a dominant coweight,
there exists $n$ big enough such that $\lambda\leq n\theta$.

\vspace{2mm}

Therefore,
$(\Gr_G)^\gamma=\lim\limits_{\longrightarrow}\overline{\Gr}_G^\lambda$,
where the limit now is taken over $(\omega_{i_\gamma}+n\theta)$,
and the proposition follows from Proposition \ref{surjectivity},
Proposition \ref{factorization}, Proposition \ref{Minuscule} and
Proposition \ref{higest coroot}.
\end{proof}

\subsection{Lattice factorization algebras v.s. Affine Kac-Moody factorization
algebras}\label{ver}
\begin{s} Let $L$ be a lattice with a symmetric bilinear form
$(\cdot,\cdot):L\times L\to \Z$ such that $(\lambda,\lambda)>0$
for all $\lambda\in L\setminus\{0\}$. For simplicity, we will
always assume that $L$ is an even lattice, i.e.
$(\lambda,\lambda)\in 2\Z$ for any $\lambda\in L$. Set
$T=L\otimes_{\Z}\G_m$. This is a torus with Lie algebra
$\frakt=L\otimes\C$ and coweight lattice $L\subset\frakt$. We
still denote $(\cdot,\cdot)$ the bilinear form on $\frakt$
obtained by the extension of the one on $L$. Let $\hat{\frakt}$ be
the Heisenberg algebra corresponding to $(\frakt,(\cdot,\cdot))$.

Choose a 2-cocycle $\varepsilon:L\times L\to \Z/2$ such that
$\varepsilon(\lambda,\mu)+\varepsilon(\mu,\lambda)=(\lambda,\mu)
\mbox{ mod } 2$. Attached to this data, there is a canonical
symmetric central extension $\tilde{T}_K$ of $T_\K$ by $\G_m$,
with a canonical splitting $i:T_\calO\to\tilde{T}_K$ (cf.
\cite{Be} and \cite{Ga} \S 6.1.1), such that the commutator
pairing $T_\K\times T_\K\to \G_m$ is given by
\begin{equation}\label{Contou-Carrere}\{\lambda(f),\mu(g)\}=\{f,g\}^{-(\lambda,\mu)}\end{equation} Here we
regard $\lambda\in L$ as morphism $(\G_m)_\K\to T_\K$ and
$\{\cdot,\cdot\}:(\G_m)_\K\times(\G_m)_\K\to \G_m$ is the
Contou-Carr\`{e}re symbol. This central extension defines an
invertible sheaf $\calL_T$ over $\Gr_T=T_\K/T_\calO$. Observe that
in our convention, $\calL_T$ corresponds to $(R_Q)^{-1}$ as in
\cite{Ga} \S 6.1.1.

Recall that as topological spaces, $(\Gr_T)_{red}=\{s^\lambda\}$
for $\lambda\in L$. Denote $\delta_\lambda$ the right $D$-module
of delta-functions at $s^\lambda$, i.e.
$\delta_\lambda=(s^\lambda)_!\C$, where $s^\lambda$ is regarded as
the map $\spec\C\to \Gr_T$. Define
\[V_L=\Gamma(\Gr_T,(\oplus_{\lambda\in
L}\delta_\lambda)\otimes\calL_T^{-1})\cong\Gamma(\Gr_T,\calL_T)^*\]
Then it is clear that as $\hat{\frakt}$-modules
\[V_L\cong\bigoplus_{\lambda\in L}\pi_\lambda\]
\end{s}
\begin{s}
If we begin with the global curve $X$, then the connected
components of $\Gr_{T,X^n}$ are labelled by
$(\lambda_1,\ldots,\lambda_n)\in L^n$, and the reduced part of
each connected component is isomorphic to $X^n$. Let
$s^{\lambda_1,\cdots,\lambda_n}$ be the corresponding section
$X^n\to \Gr_{T,X^n}$ and $\delta_{\lambda_1,\ldots,\lambda_n}$ the
corresponding sheaf of delta-functions. i.e.
$\delta_{\lambda_1,\ldots,\lambda_n}=(s^{\lambda_1,\cdots,\lambda_n})_!\omega_{X^n}$.

There are canonical line bundles $\calL_{T,X^n}$ over
$\Gr_{T,X^n}$, which are obtained from $\calL_T$ by moving points.
They satisfy the factorization property since so does the
Contou-Carr\`{e}re symbol (cf. \cite{Be} 2.3). There is another
description of $\{\calL_{T,X^n}\}$ when $X$ is complete, similar
to the ones for a simple algebraic group as in \ref{line bundle}.
Namely, there is a canonical line bundle $\calL_T$ on $Bun_{T,X}$,
and $\calL_{T,X^n}$ are the pullbacks of $\calL_T$ by
$\pi_n:\Gr_{T,X^n}\to Bun_{T,X}$. We will simply denote them by
$\calL_T$ in the following. Let $p_n:\Gr_{T,X^n}\to X^n$ be the
natural projections. Then
\[\{\mathcal
G_n:=(p_n)_*(\oplus_{(\lambda_1,\ldots,\lambda_n)\in
L^n}\delta_{\lambda_1,\ldots,\lambda_n}\otimes\calL_T^{-1})\otimes
\omega_{X^n}^{-1}\cong((p_n)_*\calL_T)^*\}\] form a factorization
algebra (cf. \cite{Be} Proposition 3.3). The vertex algebra
structure on the fiber $\mathcal G_1\otimes\C_x\cong V_L$ is
called the lattice vertex algebra. The embedding
\[(p_n)_*(\delta_{0,\ldots,0}\otimes\calL_T^{-1})\otimes
\omega_{X^n}^{-1}\to \mathcal G_n\] identifies the Heisenberg
factorization algebra as a subalgebra of the lattice factorization
algebra.
\end{s}

\begin{s}Let $G$ be the simple simply-connected algebraic group
whose Lie algebra is $\frakg$. Recall the Beilinson-Drinfeld
Grassmannian $\Gr_{G,X^n}$ from \ref{BD Grass} and the invertible
sheaf $\calL_G$ from \ref{line bundle}. For any $n$, \emph{the}
trivial $G$-bundle gives a section $e_n:X^n\to \Gr_{G,X^n}$. In
other words, $e_2=s^{0,0}$ as defined in \ref{global orbits}. On
the other hand, one has the projection $p_n:\Gr_{G,X^n}\to X^n$.
Then the collection
\[\{\V_n:=(p_n)_*(\calL_G^{\otimes
(-k)}\otimes (e_n)_!\omega_{X^n})\otimes\omega_{X^n}^{-1}\}\] form
a factorization algebra. For any $x\in X$ a closed point,
\[\V_1\otimes \C_x\cong\Gamma(\Gr_G,\calL_G^{\otimes (-k)}\otimes
\mbox{IC}_0)\cong\bbV(k\Lambda)\] as $\hat{\frakg}$-modules. Here
$\mbox{IC}_\lambda$ is the irreducible $D$-module on $\Gr_G$ whose
support is $\overline{\Gr}_G^\lambda$, and
\[\bbV(k\Lambda)=\ind_{\frakg\hat{\otimes}\calO+\C K}^{\hat{\frakg}}\C\]
is the level $k$ vacuum module, on which
$\frakg\hat{\otimes}\calO$ acts through the trivial character and
$K$ acts by multiplication by $k$. The factorization structure
endows $\V_1\otimes\C_x\cong\bbV(k\Lambda)$ with a vertex algebra
structure, which is isomorphic to the standard affine Kac-Moody
vertex algebra (cf. \cite{FB} Proposition 20.4.3 and \cite{Ga}
Theorem 5.3.1).
\end{s}

\begin{s}
Let $\{\F_n:=((p_n)_*\calL_G^{\otimes k})^*\}$. Then each $\F_n$
is a quasi-coherent sheaf on $X^n$. (However,
$(p_n)_*\calL_G^{\otimes k}$ is not quasi-coherent.) It is clear
that for any $x\in X$ closed point,
$\F_1\otimes\C_x\cong(\Gamma(\Gr_G,\calL_G^{\otimes k}))^*\cong
L(k\Lambda)$ as $\hat{\frakg}$-modules. Moreover, the
factorization property of $\calL_G$ implies that \end{s}
\begin{lem}$\{\F_n\}$ form a factorization algebra. Furthermore, there is a
natural morphism of factorization algebras $\{\V_n\}\to\{\F_n\}$
commuting with the fiberwise $\hat{\frakg}$-action.
\end{lem}
\begin{proof} It is clear that the factorization structure of
$\{\F_n\}$ comes from the factorization property of $\calL_G$ as
indicated in \ref{line bundle}. The unit is given by
$(p_1)_*\calL_G^{\otimes k}\to(p_1)_*(\calL_G^{\otimes k}\otimes
\calO_{\overline{\Gr}_{G,X}^0})\cong\calO_X$. Here,
$\overline{\Gr}_{G,X}^\lambda$ is the global analogue of
$\overline{\Gr}_G^\lambda$ (see \ref{global orbits}).

To construct a morphism $\{\V_n\}\to\{\F_n\}$ between
factorization algebras which commutes with the
$\hat{\frakg}$-action, we first construct an $\calO_{X^n}$-linear
map
\[(p_n)_*((e_n)_!\omega_{X^n})\to\omega_{X^n}\]
We need

\vspace{2mm}

\noindent\bf Claim. \rm If $f:M\to N$ is a smooth morphism of
smooth algebraic varieties. Then for any $\F$ a right $D$-module
on $M$, there is a natural morphism $Rf_*\F\to f_!\F$ in the
derived category of $\calO_N$-modules, where $f_!$ is the
push-forward of $D$-modules.

\vspace{2mm}

For any smooth algebraic variety $M$, $D_M$ denotes the sheaf of
differential operators on $M$. For $f:M\to N$ a morphism between
smooth algebraic varieties, $D_{M\to N}$ denotes the $(D_M\times
f^{-1}D_N)$-bimoulde on $M$, whose underline $\calO_M$-module
structure is isomorphic to $f^*D_N$. For $\F$ a right $D$-module
on $M$, the push-forward $f_!\F$ to $N$ is
\[f_!\F=Rf_*(\F\otimes^L_{D_M}D_{M\to N})\]

If $f:M\to N$ is a smooth morphism, $D_{M\to N}$ has a resolution
as left $D_M$-modules by
$D_M\otimes\bigwedge^{-\cdot}\Theta_{M/N}$, where $\Theta_{M/N}$
is the relative tangent sheaf. Therefore, one obtains the natural
morphism
$\F\to\F\otimes\bigwedge^{-\cdot}\Theta_{M/N}\cong\F\otimes^L_{D_M}
D_{M\to N}$. The claim is proved.

Now, since $p_ne_n=id$, by the claim, one obtains
\[(p_n)_*((e_n)_!\omega_{X^n})\to(p_n)_!((e_n)_!\omega_{X^n})\cong \omega_{X^n}\]
Together with the natural $\calO_{X^n}$-module morphism
\[(p_n)_*(\calL_G^{\otimes (-k)}\otimes (e_n)_!\omega_{X^n})\otimes\omega_{X^n}^{-1}\otimes(p_n)_*\calL_G^{\otimes k}\to(p_n)_*((e_n)_!\omega_{X^n})\otimes\omega_{X^n}^{-1}\]
one obtains the map
\[(p_n)_*(\calL_G^{\otimes (-k)}\otimes (e_n)_!\omega_{X^n})\otimes\omega_{X^n}^{-1}\otimes(p_n)_*\calL_G^{\otimes
k}\to\calO_{X^n}\] In this way, one constructs a map
$\{\V_n\}\to\{\F_n\}$. It is easy to check that it satisfies all
the required properties.
\end{proof}

\begin{rmk} (i) The factorization structure of $\{\F_n\}$ endows
$L(k\Lambda)\cong\F_1\otimes\C_x$ with a vertex algebra structure.
By the lemma above, it is realized as a quotient of the affine
Kac-Moody vertex algebra $\bbV(k\Lambda)$, and therefore,
coincides with the usual vertex algebra structure on the
integrable representation of level $k$.

(ii) Since $L(k\Lambda+\check{\nu})$ is a smooth
$\hat{\frakg}$-module, it is a module over the vertex algebra
$\bbV(k\Lambda)$ (cf. \cite{FB} Theorem 5.1.6). However, it is
known that the action of $\bbV(k\Lambda)$ on
$L(k\Lambda+\check{\nu})$ factors through $L(k\Lambda)$.
Therefore, $L(k\Lambda+\check{\nu})$ is a module over
$L(k\Lambda)$.
\end{rmk}

\subsection{The Frenkel-Kac-Segal isomorphism}\label{FKS}
Let $G$ be a simple, simply-connected algebraic group of $A,D,E$
type, with Lie algebra $\frakg$. Then the coroot lattice $R_G$
(=coweight lattice in this case) together with the normalized
invariant form on $\frakg$ is an even lattice. The torus
$T=R_G\otimes_{\Z}\G_m$ is the maximal torus of $G$. We have

\begin{lem} (i). The restriction of $\calL_G$ to $\Gr_T\subset \Gr_G$ is just
$\calL_T$ associated to the invariant form, i.e.
$\calL_T\cong\calL_G\otimes\calO_{\Gr_T}$.

(ii). The restrictions of $\calL_G$ on $\Gr_{G,X^n}$ to
$\Gr_{T,X^n}\subset \Gr_{G,X^n}$ are canonically isomorphic to
$\calL_T$ on $\Gr_{T,X^n}$. Furthermore, such restrictions are
compatible with all the factorizations.
\end{lem}
\begin{proof} We only give a rough explanation since we did not
give the precise definitions of $\calL_T$ and $\calL_G$. Indeed,
$\calL_G$ on $\Gr_G$ (resp. $\calL_G\otimes\calO_{\Gr_T}$ on
$\Gr_T$) corresponds to the inverse of the $\G_m$-central
extension of $G_\K$ (resp. $T_\K$) obtained by pullback the
central extension
\[1\to\G_m\to \mathbf{GL}(L_1(\Lambda))\to \mathbf{PGL}(L_1(\Lambda))\to 1\]
We call it the determinantal central extension. It is proved in
\cite{KV} (for the case $G=SL_2$) that the commutator pairing of
this central extension is given by formula (\ref{Contou-Carrere}).
This proves (i).

By (i), the pullback of the canonical line bundle $\calL_G$ on
$Bun_{G,X}$ (see \ref{line bundle}) to $Bun_{T,X}$  is isomorphic
to the canonical line bundle $\calL_T$ on $Bun_{T,X}$ as
introduced in 3.2.2. Then (ii) is clear.
\end{proof}

Putting things together, the geometrical form of the FKS
isomorphism is

\begin{thm} The dual of the natural morphism
$(\pi_n)_*\calL_G\to(\pi_n)_*(\calL_G\otimes\calO_{\Gr_{T,X^n}})$
gives an isomorphism of factorization algebras $\{\mathcal
G_n\}\to\{\F_n\}$.
\end{thm}
\begin{proof} By the lemma above, this is a factorization algebra
morphism. It remains to prove that this is an isomorphism, which
can be checked fiberwise, and is proved in Theorem 3.1.2.
\end{proof}

\begin{s} Finally, let us explain why Theorem \ref{Vertex operator
modules} is true. Pick up a point $x\in X$. Then the category of
$\{\mathcal G_n\}$-modules supported at $x$ is semisimple, with
all simple objects labelled by $\gamma\in L'/L$, where
$L'=\{\lambda\in L\otimes_{\Z}\mathbb Q|(\lambda,\mu)\in\Z,
\forall\mu\in L\}$ is the dual lattice (cf. \cite{Do} and
\cite{Be} Lemma 1.9 and Proposition 3.8). For each $\gamma\in
L'/L$, the corresponding module is
\[V_L^\gamma=\bigoplus_{\lambda\in\gamma+L}\pi_\lambda\]
Now let $L=R_G$ be the coroot lattice of $G$. Since
$V_{R_G}^\gamma$ is a simple $\{\mathcal G_n\}$-module, it is a
simple $\{\F_n\}$-module, and therefore is a simple
$\hat{\frakg}$-module. According to Theorem 3.1.2,
$V_{R_G}^\gamma$ is isomorphic to
$L(\Lambda+\iota\omega_{i_\gamma})$ as $\hat{\frakt}$-modules.
Thus, they must be isomorphic as $\hat{\frakg}$-modules. Or
equivalently, they are isomorphic as modules over $\{\F_n\}$.
\end{s}


\begin{thebibliography}{99}

\bibitem[BL1]{BL1}Beauville, A., Laszlo, Y.: {\it Conformal blocks and generalized theta
functions}, Commum. Math. Phys. 164, 385-419, 1994.

\bibitem[BL2]{BL2}Beauville, A., Laszlo, Y.: {\it Un lemme de
descente}, Comptes Rendus Acad. Sci. Paris 320, 335-340, 1995.

\bibitem[BLS]{BLS}Beauville, A., Laszlo, Y., Sorger, C.: {\it The Picard group of the moduli of
$G$-bundles on a curve}, Compositio Math. 112, 183¨C216, 1998.

\bibitem[Be]{Be}Beilinson, A.: {\it Langlands parameters for Heisenberg modules}, Studies in Lie theory, 51--60, Progr. Math., 243, Birkh\"{a}user Boston, Boston, MA, 2006.

\bibitem[BD1]{BD1}Beilinson, A., Drinfeld, V.: {\it Chiral
algebras}, American Mathematical Society Colloquium Publications,
51, AMS, 2004.

\bibitem[BD2]{BD2}Beilinson, A., Drinfeld, V.: {\it Quantization of Hitchin's integrable system and Hecke
eigensheaves}, Preprint, available at
www.math.uchicago.edu/~mitya/langlands

\bibitem[Bour]{Bour}Bourbaki, N.: {\it Groupes et alg\`{e}bres de Lie, Chap.
IV-VI}, Hermann, Paris, 1968.

\bibitem[DG]{DG}Demazure, M., Gabriel, P.: {\it Groupes
alg\'{e}riques}, North-Holland, 1970.

\bibitem[Do]{Do}Dong, C.: {\it Vertex algebras associated with even
lattices}, J. Algebra 161, 245-265, 1993.

\bibitem[EGAIII]{EGAIII}Grothendieck, A., Dieudonn\'{e}, J.: {\it El\'{e}ments de G\'{e}om\'{e}trie Alg\'{e}brique
III}, Publ. Math. I.H.E.S. 17, 1963.

\bibitem[EM]{EM}Evens, S., Mirkovi\'{c}, I.: {\it Characteristic cycles for the loop Grassmannian and nilpotent
orbits}, Duke Math. J. 97, no. 1, 109--126, 1999.

\bibitem[Fa]{F}Faltings, G.: {\it Algebraic loop groups and moduli spaces of
bundles}, J. Eur. Math. Soc. 5, 41-68, 2003.

\bibitem[FL]{FL}Fourier, G., Littelmann, P.: {\it Tensor product structure of affine Demazure modules and limit
constructions},  Nagoya Math. J. 182, 171--198, 2006.

\bibitem[FB]{FB}Frenkel, E., Ben-Zvi, D.: {\it Vertex algebras and algebraic
curves}, 2nd Edition, Mathematical Surveys and Monograph 88, AMS,
2004.

\bibitem[FK]{FK}Frenkel, I., Kac, V.: {\it Basic representations of affine Lie algebras and dual resonance
models}, Invent. Math. 62, 23-66, 1980.

\bibitem[Ga]{Ga}Gaitsgory, D.: {\it Notes on 2D conformal field theory and string
theory}, in {\it Quantum fields and strings: a course for
mathematicians}, Vol. 2, 1017-1089, AMS, 1999.

\bibitem[GKM]{GKM}Goresky, M., Kottwitz, R., Macphserson, R.: {\it Homology of affine Springer fibers in the unramified case}, Duke Math. J. 121, no. 3, 509--561, 2004.

\bibitem[Kac]{Kac}Kac, V.: {\it Infinite-dimensional algebras, Dedekind's $\eta$-function, classical M\"{o}bius function and the very strange
formula}, Adv. in Math. (30), 85-136, 1978.

\bibitem[KKLW]{KKLW} Kac, V., Kazhdan, D., Lepowsky, J., Wilson, R.: {\it
Realization of the basic representations of the Euclidean Lie
algebras}, Adv. in Math. 42, no. 1, 83--112, 1981.

\bibitem[KP]{KP}Kac, V., Peterson, D.: {\it $112$ constructions of the basic representation of the loop group of $E\sb 8$}, Symposium on anomalies, geometry, topology (Chicago, Ill., 1985), 276--298, World Sci. Publishing, Singapore, 1985.

\bibitem[KV]{KV}Kapranov, M., Vasserot, E.: {\it Formal loops II: a local Riemann-Roch theorem for determinantal
gerbes}, Ann. Sci. ¨¦cole Norm. Sup. (4) 40, no. 1, 113--133,
2007.

\bibitem[KL]{KL}Kazhdan, D., Lusztig, G.: {\it Fixed point varieties on affine flag
manifolds}, Israel J. math. 62(2), 129-168, 1988.

\bibitem[Ku]{Ku}Kumar, S.: {\it Demazure character formula in arbitrary Kac-Moody
setting}, Invent. Math. 89, 395-423, 1987.

\bibitem[LS]{LS}Laszlo, Y., Sorger, C.: {\it The line bundles on the moduli of parabolic $G$-bundles over curves and their
sections}, Ann. Sci. \'{E}cole Norm. Sup. (4), 30(4), 499-525,
1997.

\bibitem[MOV]{MOV}Malkin, A., Ostrik, V., Vybornov, M.: {\it The minimal degenration singularities in the affine
Grassmannians}, Duke Math. J. 126, No. 2, 233-249, 2005.

\bibitem[Ma]{Ma}Mathieu, O.: {\it Formules de caract\`{e}res pour les alg\`{e}bres de Kac-Moody
g\'{e}n\'{e}rales}, Ast\'{e}risque, 159-160, 1988.

\bibitem[MR]{MR}Mehta, V.B., Ramanathan, A.: {\it Frobenius splitting and cohomology vanishing for Schubert
varieties}, Ann. Math. 122(1), 27-40, 1985.

\bibitem[MV]{MV}Mirkovi\'{c}, I., Vilonen, K.: {\it Geometric Langlands duality and representations of algebraic groups over commutative
rings}, Ann. of Math. (2) 166 (2007), no. 1, 95--143..

\bibitem[Se]{Se}Segal, G.: {\it Unitary representations of some infinite-dimensional
groups}, Comm. Math. Phys. 80, no. 3, 301--342, 1981.

\bibitem[Zhu]{Zhu}Zhu, X.: {\it Basic representations via affine Springer
fiber}, in preparation.
\end{thebibliography}
\end{document}